\documentclass{ps}
\usepackage{amsmath,amsfonts}
\usepackage{color}
\usepackage[applemac]{inputenc}
\usepackage{mathrsfs}
\usepackage{mathtools}

\def \O{\mathcal{O}}

\def\leftB{[\![}
\def\rightB{]\!]}

\newcommand \A[1]{{\bf (#1)}}

\def\O{{\cal{O}}}

\def\F{{\cal F}}

\def\P{{\mathbb{P}}  }

\def\det{{\rm{det}}}
\def\tr{{\rm{tr}}}
\def\bint#1^#2{\displaystyle{\int_{#1}^{#2}}}
\def\bsum#1^#2{\displaystyle{\sum_{#1}^{#2}}}

\def\xdt_#1{X_#1(\Delta t)}

\newtheorem{THM}{Theorem}[section]

\newtheorem{PROP}{Proposition}[section]

\newtheorem{COROL}{Corollary}[section]

\newtheorem{LEMME}{Lemma}[section]
\newtheorem{REM}{Remark}[section]

\def \R{\mathbb{R}}
\def \N{\mathbb{N}}
\def \E{\mathbb{E}}
\def \argmin{\arg~\min}

\newtheorem{thm}{Theorem}[section]

\newtheorem{lem}[thm]{Lemma}

\DeclarePairedDelimiter \ceil{\lceil}{\rceil}

\begin{document}
\title{A Multi-step Richardson-Romberg extrapolation method for stochastic approximation}
\author{N. Frikha}\address{LPMA, Universit\'e Paris Diderot, 5 rue Thomas Mann 75013 Paris, frikha@math.univ-paris-diderot.fr}
\author{L. Huang}\address{LPMA, Universit\'e Paris Diderot, 5 rue Thomas Mann 75013 Paris, huang@math.univ-paris-diderot.fr}
\date{\today}
\begin{abstract} We obtain an expansion of the implicit weak discretization error for the target of stochastic approximation algorithms introduced and studied in \cite{Frikha2013}. This allows us to extend and develop the Richardson-Romberg extrapolation method for Monte Carlo linear estimator (introduced in \cite{tala:tuba:90} and deeply studied in \cite{GPag:07}) to the framework of stochastic optimization by means of stochastic approximation algorithm. We notably apply the method to the estimation of the quantile of diffusion processes. Numerical results confirm the theoretical analysis and show a significant reduction in the initial computational cost.
\end{abstract}
\maketitle

\section{Statement of the Problem}
\label{stat:problem:sec}
The aim of this paper is to combine a multistep Richardson-Romberg extrapolation method with stochastic approximation (SA) algorithms which are recursive simulation based procedures commonly used in the framework of stochastic optimization. Introduced by Robbins and Monro \cite{Robbins1951}, SA algorithms aims at computing a zero of a continuous function $h: \R^d \rightarrow \R^d$ which is unknown to the experimenter but can only be estimated through experiments. In this general context, the function $h$ writes $h(\theta):= \E[H(\theta,U)]$ where $H: \R^d \times \R^q\rightarrow \R^d$ and $U$ is a $\R^q$-valued random vector. To estimate a zero of $h$, one devises the following recursive algorithm
\begin{equation}
\label{RM:proc}
\theta_{p+1} = \theta_p - \gamma_{p+1}H(\theta_p,U^{p+1}), \ p\geq0
\end{equation}

\noindent where $(U^{p})_{p\geq1}$ is an $i.i.d.$ sequence of random variables with the same law as $U$ defined on a probability space $(\Omega,\mathcal{F}, \P)$, $\theta_0$ is independent of the innovation of the algorithm with $\E[|\theta_0|^2]< + \infty$ and $\gamma=(\gamma_p)_{p\geq1}$ is a deterministic and decreasing sequence of non-negative steps satisfying the usual assumption
\begin{equation}
\label{step:rm:assump}
\sum_{p\geq1} \gamma_p = + \infty, \ \ \mbox{ and } \ \ \sum_{p\geq1} \gamma^{2}_p < + \infty.
\end{equation}
 
 When the function $h$ is the gradient of a convex potential, the recursive procedure \eqref{RM:proc} is a stochastic gradient algorithm. Indeed replacing $H(\theta_p, U^{p+1})$ by $h(\theta_p)$ in \eqref{RM:proc} leads to the usual deterministic descent gradient procedure. 
 
 In many applications, notably in computational finance, the sequence of random vectors $(U^{p})_{p\geq1}$ is not directly simulatable (at a reasonable cost) and can only be approximated by another sequence of easily simulatable random vectors $((U^{n})^{p})_{p\geq1}$, $n>0$, where $U^{n}$ (weakly or strongly) approximates $U$ as $n\rightarrow +\infty$ with a standard weak discretization error (or bias) $\E[f(U^{n})]-\E[f(U)]$ that can be expanded in powers of $n^{-\alpha}$, $\alpha>0$, for a specific class of functions $f \in \mathcal{C}$. One typical situation is when $U=X_T$, $X:=(X_t)_{t\in [0,T]}$ being a $q$-dimensional diffusion process solution of a stochastic differential equation (SDE) and $U^{n} = X^{n}_T$ where $X^{n}:=(X^{n}_t)_{t\in [0,T]}$ stands for its standard Euler-Maruyama discretization scheme with time step $\Delta = T/n$, $n\in \N^{*}$.
 
%

 Since we are interested in the computation of the zero $\theta^*$ of $h$ given by $h(\theta):=\E[H(\theta,U)]$ where $H: \R^d \times \R^q\rightarrow \R^d$ and the function $h$ is generally neither known nor computable since the random variable $U$ cannot be easily simulated, estimating $\theta^*$ by devising directly the recursive scheme \eqref{RM:proc} is not possible. Therefore, two steps are needed to compute $\theta^*$:
\begin{trivlist}
\item[-] the first step consists in approximating the zero $\theta^*$ of $h$ by the zero $\theta^{*,n}$ of the function $h^{n}$ defined by $h^{n}(\theta):=\E[H(\theta, U^{n})]$, $\theta \in \R^d$. It induces \emph{an implicit discretization error} which writes 
$$
\mathcal{E}_{D}(n):= \theta^{*}-\theta^{*,n}.
$$ 
%
 
\noindent Under mild assumptions on $h$ and $h^{n}$, it is proved in \cite{Frikha2013} that $\theta^{*,n}$ converges to $\theta^*$ as $n$ goes to infinity. Moreover, if the \emph{standard weak discretization error} is of order $n^{-\alpha}$, $\alpha \in (0,1)$, that is $\forall \theta \in \R^d, \ h^{n}(\theta)-h(\theta) = \Lambda^{0}_1(\theta) n^{-\alpha} + o(n^{-\alpha})$, with $\Lambda^{0}_1:\R^d \rightarrow \R^d$, then (under additional mild assumptions) this rate of convergence transfers to the \emph{implicit discretization error} that is $\mathcal{E}_{D}(n) = \Theta_1n^{-\alpha} + o(n^{-\alpha})$ for some $\Theta_1 \in \R^d$.


\item[-] the second step consists in approximating $\theta^{*,n}$ using $M\in \N^{*}$ steps of the following SA scheme
\begin{equation}
\label{RM}
\theta^{n}_{p+1} = \theta^{n}_{p} - \gamma_{p+1} H(\theta^{n}_{p}, (U^{n})^{p+1}), \ p \in \leftB 0, M-1\rightB,
\end{equation}

\noindent where $((U^{n})^{p})_{p\in \leftB1,M\rightB}$ is an i.i.d. sequence of random variables with the same law as $U^{n}$, $\theta^{n}_0$ is independent of the innovation of the algorithm with $\sup_{n\geq1}\E|\theta^{n}_0|^2<+\infty$ and $\gamma=(\gamma_{p})_{p \geq 1}$ is a sequence of non-negative deterministic and decreasing steps satisfying \eqref{step:rm:assump}.
This induces a \emph{statistical error} which writes 
$$
\mathcal{E}_S(n, M):= \theta^{*,n} - \theta^{n}_{M}.
$$

Regarding the \emph{statistical error}, it is well-known that under mild assumptions the Robbins-Monro theorem guarantees that for each $n\in \N^{*}$, $\lim_{M\rightarrow + \infty}\mathcal{E}_S(n, M) = 0$. Moreover, under additional technical assumptions, a central limit theorem (CLT) holds at rate $\gamma^{-1/2}(M)$ that is $\gamma^{-1/2}(M)\mathcal{E}_S(n, M) $ converges in distribution to a normally distributed random variable. The reader may also refer to \cite{fri:men:12} and \cite{fat:fri:13} for some recent developments on non-asymptotic deviation bounds for the statistical error.
\end{trivlist}

The global error between $\theta^{*}$, the quantity to estimate, and its implementable approximation $\theta^{n}_M$ can be decomposed as follows:
\begin{align*}
\mathcal{E}_{glob}(n, M) & = \theta^{*} - \theta^{*,n} + \theta^{*,n}- \theta^{n}_M\\
& := \mathcal{E}_{D}(n) + \mathcal{E}_S(n, M).
\end{align*} 


The first aim of this paper is to prove the existence of an expansion for the \emph{implicit discretization error}, that is, under mild assumptions (see Section \ref{main:result:sec}) on $h$ and $h^{n}$, $\mathcal{E}_{D}(n)$ can be expanded as follows  
\begin{equation}
\label{impdisc:exp}
\forall R \in \N^{*}, \ \ \theta^{*,n} - \theta^* = \frac{C_1}{n^{\alpha}} + \cdots + \frac{C_R}{n^{\alpha R}} + o\left(\frac{1}{n^{\alpha R}}\right)
\end{equation}

\noindent where $(C_1, \cdots, C_R) \in (\R^d)^R$. Then taking advantage of \eqref{impdisc:exp} we devise a multistep Richardson-Romberg extrapolation method for stochastic optimization by means of stochastic approximation algorithm. The principle of Richardson-Romberg extrapolation is to reduce the bias produced by the \emph{implicit discretization error} by combining two estimators with different step size. To be more precise, one considers the two following weights $w_1 = (-1/(2^{\alpha}-1)) I_d$ and $w_2 = (2^{\alpha}/(2^{\alpha}-1)) I_d$, $I_d$ is the identity matrix of dimension $d$ and the Richardson-Romberg SA estimator
$$
\Theta^{n, 2n}_{M} = w_1 \theta^{n}_{M} + w_2 \theta^{2n}_{M} 
$$

\noindent where $(\theta^{2n}_M,\theta^{n}_M)$ is obtained using $M$ steps of two SA schemes devised with the i.i.d. sequence $((U^{2n},U^{n})^{p})_{p \in \leftB 1, M\rightB}$ of random variables with the same law as $(U^{2n}, U^{n})$. Under standard assumptions, this linear combination of SA estimators $a.s.$ converges to the target $w_1 \theta^{*,n} + w_2 \theta^{*,2n}$ as the number of steps $M$ goes to infinity. The key observation is that this new target satisfies the following implicit error expansion of order 2
$$
 w_1 \theta^{*,n} + w_2 \theta^{*,2n}  - \theta^* = -\frac{C_2}{2^{\alpha}} \frac{1}{n^{2\alpha}} + o\left( \frac{1}{n^{2\alpha}} \right).
$$

Moreover, in the spirit of \cite{GPag:07}, we show how to control the asymptotic $L^{1}(\P)$-norm of the distance between the new estimator $\Theta^{n, 2n}_{M}$ and its target $w_1\theta^{*,n} +w_2 \theta^{*,2n}$ as $n$ goes to infinity. Then, it is natural to iterate this extrapolation to obtain a new SA estimator with an implicit discretization error of order $n^{-\alpha R}$ for any $R\in \N^{*}$. This extension called multi-step Richardson-Romberg extrapolation is deeply investigated in \cite{GPag:07} for Monte Carlo linear estimator in the framework of discretization of diffusion processes.

The aim of this paper is to investigate the Richardson-Romberg SA method. Our purpose is to show that the principle of multi-step Richardson-Romberg extrapolation for Monte Carlo linear estimator can be extended to the framework of stochastic optimization by means of SA algorithm. We notably prove that the new estimator outperforms the standard SA estimator in terms of computational cost.

The paper is organized as follows: in Section 2 we provide an expansion of the implicit discretization error in powers of $n^{-\alpha}$ under mild assumptions. Then we take advantage of this expansion to propose a multi-step Richardson-Romberg method by means of SA. In Section 3 is presented an illustration of the method to the estimation of the quantile of a stochastic differential equation (SDE) driven by a stable process. In Section 4 numerical results are carried out to confirm the theoretical analysis. Finally, Section 5 is devoted to theoretical results which are useful throughout the paper.

\section{Main results}
\label{main:result:sec}
This section is divided in two parts. In the first one we obtain a general result concerning the expansion of the implicit discretization error. In the second one, we take advantage of this result to develop a Richardson-Romberg extrapolation method for stochastic optimization by means of SA algorithms.

\subsection{Expansion of the implicit discretization error}

We first provide a result concerning the convergence of the sequence $(\theta^{*,n})_{n\geq1}$ towards $\theta^*$. For a proof the reader may refer to \cite{Frikha2013}.
\begin{PROP}
\label{prop:conv:disc}
For all $n \in \N^*$, assume that $h$ and $h^n$ satisfy the mean reverting assumption:
\begin{equation*}
\forall \theta \neq \theta^*, \ \langle \theta-\theta^*, h(\theta)\rangle >0 \ \ \mbox{and } \ \forall \theta \neq \theta^{*,n}, \  \langle \theta-\theta^{*,n}, h^{n}(\theta)\rangle >0.
\end{equation*}

 Moreover, suppose that $(h^{n})_{n\geq1}$ converges locally uniformly towards $h$. Then, one has
$$
\theta^{*,n} \rightarrow \theta^* \ \ \mbox{as} \ \ n\rightarrow + \infty.
$$
\end{PROP}

 Here we will investigate an expansion of the error term $\theta^{*,n}-\theta^*$ in powers of $n^{-\alpha}$. Through the document, we will refer to \textbf{[H-k]} the following set of assumptions:
\begin{enumerate}
\item For all $\theta \in \R^d$,
\begin{equation}\label{erreur_discretisation}
h(\theta) - h^n(\theta) = \frac{\Lambda^0_1(\theta)}{n^{\alpha}} + \cdots +  \frac{\Lambda^0_k(\theta)}{n^{\alpha k}} + o\left( \frac{1}{n^{\alpha k}}\right) .
\end{equation}

\item $h,h^n\in \mathcal C^{k}(\R^d,\R^d)$ and for all $l \leq k-1$, for all $\theta \in \R^d$,

\begin{equation}\label{dev_diff_h} 
D^lh^n(\theta) - D^lh(\theta) = \frac{\Lambda^l_1(\theta)}{n^{\alpha}} + \cdots + \frac{\Lambda^l_{k-l}(\theta)}{n^{\alpha (k-l)}} + o\left( \frac{1}{n^{\alpha (k-l)}}\right)
\end{equation}

\noindent where for all $\theta\in \R^d$, $\Lambda^{l}_1(\theta), \cdots, \Lambda^{l}_{k-l}(\theta)$ and $o(n^{-\alpha(k-l)})$ are multilinear maps from $(\mathbb{R}^d)^{l}$ to $\mathbb{R}^d$.

\item For all $l \in \leftB 1, k \rightB$, $(D^{l} h^n)_{n\geq1}$ converges locally uniformly towards $D^{l}h$. 

\item $Dh(\theta^*)$ is invertible.
\end{enumerate}

\begin{PROP}\label{premier_pas}
Assume that $\theta^{*,n} \rightarrow \theta^*$ as $n\rightarrow + \infty$. Under \textbf{[H-1]}, one has
\begin{equation}\label{initialisation}
 n^{\alpha} \left( \theta^{*,n} - \theta^* \right) \underset{n \rightarrow \infty}{\longrightarrow} Dh(\theta^*)^{-1}\Lambda^0_1(\theta^*).
\end{equation}
\end{PROP}

\begin{proof}
Observe that one has $ h^n(\theta^{*,n}) - h^n(\theta^*) = - h^n(\theta^*)=h(\theta^{*}) - h^n(\theta^*)$. On the one hand, writing Taylor's formula with integral remainder yields:
\begin{equation}
\label{taylor:first:order}
h^n(\theta^{*,n}) - h^n(\theta^*) = \int_0^1 dt Dh^n(t\theta^{*,n}+ (1-t)\theta^*)(\theta^{*,n}- \theta^*).
\end{equation}
On the other hand, from the discretization error, we have $ h(\theta^{*}) - h^n(\theta^*) = \Lambda^0_1(\theta^*) n^{-\alpha} + o\left(n^{-\alpha}\right)$.

Since $\theta^{*,n}  \underset{n \rightarrow \infty}{\longrightarrow} \theta^*$, $Dh(\theta^*)$ is invertible, and $(Dh^n)_{n\geq1}$ converges uniformly locally to $Dh$, for $n$ large enough, the matrix $ \int_0^1Dh^n(t\theta^{*,n}+ (1-t)\theta^*)dt$ is invertible. Multiplying both sides of \eqref{taylor:first:order} by $n^{\alpha}$ finally yields
$$
n^{\alpha} \left(\theta^{*,n} - \theta^* \right)= \left(\int_0^1Dh^n(t\theta^{*,n}+ (1-t)\theta^*)dt \right)^{-1} \left( \Lambda^0_1(\theta^*)+ o(1)\right) \underset{n \rightarrow \infty}{\longrightarrow} Dh(\theta^*)^{-1}\Lambda^0_1(\theta^*).
$$
\end{proof}

Let us note that Proposition \ref{premier_pas} provides a first order expansion of $\theta^{*,n}-\theta^*$, that is $\theta^{*,n}-\theta^*= C_1 n^{-\alpha}+o(n^{-\alpha})$. We now give a generalization of this first result.
\begin{THM}\label{implic:discret:err:main_result}
Assume that $\theta^{*,n}\rightarrow \theta^*$, $n\rightarrow +\infty$, and that \textbf{[H-p]} holds for some $p\in \N^{*}$. Then, $\theta^{*,n}- \theta^*$ has an expansion up to order $p$, 
that is, the following expansion holds:
$$\theta^{*,n}- \theta^* = \frac{C_1}{n^{\alpha}} + \cdots + \frac{C_p}{n^{\alpha p}} + o\left( \frac{1}{n^{\alpha p}}\right).$$
\end{THM}

\begin{proof} If \textbf{[H-p]}, $p \in \N^{*}$, holds then Proposition \ref{premier_pas} gives a first order expansion for $\theta^{*,n}- \theta^*$. We now prove the inductive step that is if $\theta^{*,n}-\theta^*$ has an expansion of order $k-1$ then an expansion holds at order $k$, for $k\leq p$. The basic idea does not change from the previous computation.
From the development of the discretization error, we have:

\begin{equation}
h(\theta^*) - h^n(\theta^*) = \frac{\Lambda_1^0(\theta^*)}{n^{\alpha}} + \cdots +  \frac{\Lambda_k^0(\theta^*)}{n^{\alpha k}} + o\left( \frac{1}{n^{\alpha k}}\right) .
\end{equation}

On the other hand, we write a Taylor's expansion of $h^n$ up to the same order $k-1$:

\begin{equation}\label{dev_h}
h^n(\theta^{*,n}) - h^n(\theta^*) = Dh^n(\theta^*)(\theta^{*,n}- \theta^*) + \cdots + \frac{1}{(k-1)!} D^{k-1}h^n(\theta^*)(\theta^{*,n} - \theta^*)^{(k-1)} + R_{k-1}^n(\theta^{*,n} - \theta^*),
\end{equation}

\noindent with the remainder in integral form satisfying:
\begin{align*}
R_{k-1}^n(\theta^{*,n} - \theta^*) & = \int_0^1 \frac{(1-t)^{k-1}}{(k-1)!} D^{k} h^n ( t \theta^{*,n}+ (1-t)\theta^*) (\theta^{*,n} - \theta^*)^{(k)} dt  = \frac{1}{k!}D^{k}h(\theta^*)(\theta^{*,n}-\theta^*)^{(k)}  + o\left(\frac{1}{n^{\alpha k}}\right)
\end{align*}

\noindent where we used that $(D^{k} h^n)_{n\geq1}$ converges locally uniformly to $D^{k}h$, $k\in\leftB1,p \rightB$, and $\theta^{*,n}-\theta^* = \O(n^{-\alpha})$ for the last equality. Let us note that for $l\in \leftB1,k\rightB$, $D^{l}h(\theta^*)$ (as $\Lambda^{l}_{j}(\theta^*)$, $j=1, \cdots,k-l$) is a multilinear maps from $(\R^d)^l$ to $\R^d$. The expansions \eqref{dev_diff_h} allow us to replace the derivatives of $h^n$ by the derivatives of $h$ in \eqref{dev_h} at the cost of an error term, that is:
%

\begin{align*}
h^n(\theta^{*,n}) - h^n(\theta^*) &= Dh(\theta^*)(\theta^{*,n}- \theta^*) + \left(\frac{\Lambda^1_1(\theta^*)}{n^{\alpha}} + \cdots + \frac{\Lambda^1_{k-1}(\theta^*)}{n^{\alpha (k-1)}} + o\left( \frac{1}{n^{\alpha (k-1)}}\right)\right)(\theta^{*,n} - \theta^*) \\
&+ \cdots + \frac{1}{(k-1)!} \left( D^{k-1}h(\theta^*) + \frac{\Lambda^{k-1}_1(\theta^*)}{n^{\alpha}}+ o\left( \frac{1}{n^{\alpha}}\right) \right)(\theta^{*,n} - \theta^*)^{(k-1)}\\
&+ \frac{1}{k!}D^kh(\theta^*)(\theta^{*,n}-\theta^*)^{(k)}+o\left(\frac{1}{n^{\alpha k}}\right).
\end{align*}

Since $h^n(\theta^{*,n}) - h^n(\theta^*) =-h^n(\theta^*) = h(\theta^*) - h^n(\theta^*) $ and $Dh(\theta^*)$ is invertible, the previous equality implies

\begin{align*}
& \frac{Dh(\theta^*)^{-1}\Lambda_1^0(\theta^*)}{n^{\alpha}} + \cdots +  \frac{Dh(\theta^*)^{-1}\Lambda_k^0(\theta^*)}{n^{\alpha k}} + o\left( \frac{1}{n^{\alpha k}}\right) = \\
& \theta^{*,n}- \theta^* + \left(\frac{Dh(\theta^*)^{-1}\Lambda^1_1(\theta^*)}{n^{\alpha}}+ \cdots + \frac{Dh(\theta^*)^{-1}\Lambda^1_{k-1}(\theta^*)}{n^{\alpha (k-1)}} + o\left( \frac{1}{n^{\alpha (k-1)}}\right) \right)(\theta^{*,n}- \theta^*) \\
&+ \cdots +  \frac{1}{(k-1)!}\left( Dh(\theta^*)^{-1}D^{k-1}h(\theta^*)+ \frac{Dh(\theta^*)^{-1}\Lambda^{k-1}_1(\theta^*)}{n^{\alpha}}  + o\left(\frac{1}{n^{\alpha}}\right) \right)(\theta^{*,n}- \theta^*)^{(k-1)}\\
&+ \frac{1}{k!}Dh(\theta^*)^{-1}D^kh(\theta^*)(\theta^{*,n}- \theta^*)^{(k)} + o\left( \frac{1}{n^{\alpha k}}\right).
\end{align*}

The last equation should be seen as a "bootstrap" for $\theta^{*,n}- \theta^*$, that is:

\begin{align}
\theta^{*,n}- \theta^* & =
\frac{Dh(\theta^*)^{-1}\Lambda^0_1(\theta^*)}{n^{\alpha}} + \cdots +  \frac{Dh(\theta^*)^{-1}\Lambda^0_k(\theta^*)}{n^{\alpha k}} + o\left( \frac{1}{n^{\alpha k}}\right) \nonumber\\
& - \left(\frac{Dh(\theta^*)^{-1}\Lambda^1_1(\theta^*)}{n^{\alpha}} + \cdots + \frac{Dh(\theta^*)^{-1}\Lambda^1_{k-1}(\theta^*)}{n^{\alpha (k-1)}} + o\left( \frac{1}{n^{\alpha (k-1)}}\right) \right)(\theta^{*,n}- \theta^*) \nonumber\\
& - \cdots \nonumber \\
& - \frac{Dh(\theta^*)^{-1}}{(k-1)!} \left( D^{k-1}h(\theta^*) + \frac{\Lambda^{k-1}_1(\theta^*)}{n^{\alpha}}
+o\left( \frac{1}{n^{\alpha}}\right) \right) (\theta^{*,n}- \theta^*)^{(k-1)}\nonumber\\
&- \frac{1}{k!}Dh(\theta^*)^{-1}D^kh(\theta^*)(\theta^{*,n}- \theta^*)^{(k)} + o\left( \frac{1}{n^{\alpha k}}\right), \label{bootstrap}
\end{align}

The idea now is to plug the expansion of $\theta^{*,n}- \theta^*$ in the right hand side of \eqref{bootstrap} and check that the first remainder term comes at order $o(n^{-\alpha k})$. It is clear that on the first line the remainder term is of order $o(n^{-\alpha k})$. Moreover, for any $l \in \leftB2,k\rightB$, the generic $l$-th term writes in the $i$-th component:

\begin{align}
& \frac{1}{l!} \left( \left(D^lh(\theta^*) + \frac{\Lambda_{1}^{l}(\theta^*)}{n^{\alpha}} + \cdots +\frac{\Lambda_{k-l}^l(\theta^*)}{n^{\alpha (k-l)}} + o\left( \frac{1}{n^{\alpha (k-l)}}\right) \right) (\theta^{*,n}-\theta^*)^{(l)} \right)_{i} \nonumber \\
& = \sum_{i_1+ \cdots + i_d = l } \frac{1}{i_1! \cdots i_d!}\Lambda_{i_1,\cdots,i_d} (\theta^{*,n}-\theta^*)_1^{i_1} \times \cdots\times (\theta^{*,n}-\theta^*)_d^{i_d}  + o\left( \frac{1}{n^{\alpha (k-l)}}\right)(\theta^{*,n}-\theta^*)_1^{i_1} \times \cdots\times (\theta^{*,n}-\theta^*)_d^{i_d} \nonumber \\
& = \sum_{i_1+ \cdots + i_d = l } \frac{1}{i_1! \cdots i_d!} \Lambda_{i_1,\cdots,i_d} (\theta^{*,n}-\theta^*)_1^{i_1} \times \cdots\times (\theta^{*,n}-\theta^*)_d^{i_d}  + o\left( \frac{1}{n^{\alpha k}}\right) \label{onecomponent:sum}
\end{align}

\noindent where $\Lambda_{i_1,\cdots,i_d} = \frac{\partial^{l}h_i}{\partial \theta^{i_1}_{1} \cdots \partial \theta^{i_d}_{d}}(\theta^*) + \frac{(\Lambda_{1}^{l}(\theta^*))_i}{n^{\alpha}} + \cdots +\frac{(\Lambda_{k-l}^l(\theta^*))_{i}}{n^{\alpha (k-l)}}$ with $(\Lambda^{l}_{j}(\theta^*))_{i}$ for $j \in \leftB1,k-l\rightB$ satisfying $\frac{\partial^{l}h_i}{\partial \theta^{i_1}_{1} \cdots \partial \theta^{i_d}_{d}}(\theta^*) - \frac{\partial^{l}h^{n}_i}{\partial \theta^{i_1}_{1} \cdots \partial \theta^{i_d}_{d}}(\theta^*) = (\Lambda^{l}_{1}(\theta^*))_{i}/n^{\alpha} + \cdots + (\Lambda^{l}_{k-l}(\theta^*))_{i} /n^{\alpha (k-l)} + o(1/n^{\alpha (k-l)})$ and where we used that $(\theta^{*,n}-\theta^*)_1^{i_1} \times \cdots\times (\theta^{*,n}-\theta^*)_d^{i_d} = \O(1/n^{\alpha l})$ for the last equality. Now, replacing $(\theta^{*,n}-\theta^*)_i$ by its expansion, we observe that the generic term in \eqref{onecomponent:sum} satisfies 
\begin{align*}
& \Lambda_{i_1,\cdots,i_d} \left( \frac{C^{1}_1}{n^{\alpha}} + \cdots + \frac{C^{1}_{k-1}}{n^{\alpha (k-1)}} + o \left( \frac{1}{n^{\alpha (k-1)}}\right) \right)^{i_1}
\times \cdots \times \left( \frac{C^{d}_1}{n^{\alpha}} + \cdots + \frac{C^{d}_{k-1}}{n^{\alpha (k-1)}} + o \left( \frac{1}{n^{\alpha (k-1)}}\right) \right)^{i_d} \\
& = \Lambda_{i_1,\cdots,i_d} \left( \frac{\tilde{C}}{n^{\alpha l}} + \cdots + o\left(\frac{1}{n^{\alpha (k+l-2)}}\right) \right) \\
& = \Lambda_{i_1,\cdots,i_d} \left( \frac{\tilde{C}}{n^{\alpha l}} + \cdots + o\left(\frac{1}{n^{\alpha k}}\right) \right)
\end{align*}

\noindent where $\tilde{C} = (C^{1}_1)^{i_1} \times \cdots \times (C^{d}_1)^{i_d}$. We clearly see that the expression above yields an expansion in powers of $n^{-\alpha}$ with a remainder at order $o(n^{-\alpha k})$. Formally, as the power in the expansion \eqref{dev_diff_h} goes down, the power in the derivatives grows, compensating exactly and giving the right order in the remainder.

Finally, we expand the previous equation and group together the different terms with respect to the power of $n^{-\alpha}$. As we observed above, the remainder term is at order $o\left(n^{-\alpha k}\right)$, because of the compensation between the power in the expansion \eqref{dev_diff_h} and the order of the Taylor expansion. This completes the proof.
\end{proof}
\subsection{Multi-step Richardson-Romberg extrapolation for stochastic approximation}
Multi-step Richardson-Romberg extrapolation was successfully applied in the context of Monte Carlo linear estimator for the computation of $\E[f(X_T)]$, where $f: \R^d \rightarrow \R$ (with possible extension to the case of path-dependent options) and $X$ is the (unique) strong solution to a SDE, see \cite{GPag:07}. In this section, we propose a multi-step Richardson-Romberg SA estimator with a control of the statistical error. We proceed as follows. Let $R\geq2$ be an integer. To devise a SA estimator whose target has an implicit discretization error of order $n^{-\alpha R}$ as $n\rightarrow + \infty$, we introduce a sequence of $R$ random vectors $\left\{U^{rn}, r \in \leftB 1, R \rightB \right\}$, $n\in \N^{*}$. Throughout this section we will assume that this sequence satisfies $U^{rn} \overset{\P}{\longrightarrow} U^{r}$ as $n\rightarrow + \infty$ with $U^{r} \overset{d}{=} U$, $r\in \leftB1,R\rightB$, all variables being defined on the same probability space. If assumption \textbf{[H-R]} holds then for all $r\in \leftB 1, R\rightB$ one gets
$$
\theta^{*,rn} = \theta^{*} + \sum_{p=1}^{R-1} \frac{C_p}{r^{\alpha p}} \frac{1}{n^{\alpha p}} + \frac{C_R}{r^{\alpha R}} \frac{1}{n^{\alpha R}} \left(1+ \epsilon_{r}(n) \right) 
$$ 

\noindent with $\epsilon_r(n) \rightarrow 0$ as $n \rightarrow + \infty$. Then, one defines the Vandermonde $Rd \times (R-1)d$ matrix 
$$
V = \left[ \frac{I_d}{r^{\alpha p}}\right]_{1\leq r \leq R, 1 \leq p \leq R-1}
$$

\noindent and the extended $Rd \times d$ unit matrix $\bold{I}=\left(I_d, \cdots, I_d\right)^{T}$ where $I_d$ is the identity matrix of dimension $d$. Now we write
\begin{equation}
\label{first:step:richardromberg}
\begin{pmatrix}
& \vdots &\\
& \theta^{*,r n} &\\
& \vdots & \\
\end{pmatrix}_{1 \leq r \leq R}
=
\bold{I} \theta^{*} + V \begin{pmatrix}
& \vdots &\\
&  \frac{C_{r}}{n^{\alpha r}} &\\
& \vdots & \\
\end{pmatrix}_{1 \leq r \leq R-1} 
+ 
 \begin{pmatrix}
& \vdots &\\
& \frac{C_R}{r^{\alpha R}} \frac{1}{n^{\alpha R}} \left(1+ \epsilon_r(n) \right)&\\
& \vdots & \\
\end{pmatrix}_{1 \leq r \leq R} .
\end{equation}

 We consider the $ Rd \times d$ weight matrix $\bold{w}= (\bold{w}_1, \cdots, \bold{w}_R)^{T}$, $\bold{w}_i$ being a $d\times d$ matrix for $i \in \leftB 1, R \rightB$ satisfying 
\begin{equation}
\label{twoeq:vandermonde:syst}
\bold{w}^{T} \bold{I} = I_d \ \ \ \mbox{and} \ \ \ \bold{w}^{T} V = 0_{d\times d(R-1)}
\end{equation}

\noindent which is equivalent to 
\begin{equation}
\label{vandermonde:sys}
 \tilde{V} \bold{w} = E_1
\end{equation}

\noindent with $E_1=(I_d, 0_{d\times d(R-1)})^{T}$ and $\tilde{V}$ is the Vandermonde matrix defined by
$$
\tilde{V} = \begin{pmatrix}
  I_{d} & I_d & \cdots & I_d \\
  I_{d} & \frac{I_d}{2^{\alpha}}      &  \cdots & \frac{I_d}{R^{\alpha}}  \\
  \vdots & \vdots & \cdots & \vdots \\
  I_d  & \frac{I_d}{2^{(R-1)\alpha}}  & \cdots & \frac{I_d}{R^{(R-1)\alpha}} 
\end{pmatrix}.
$$

Thanks to Cramer's rule, the solution $\bold{w}$ to \eqref{vandermonde:sys} is explicitly given by
\begin{equation}
\label{sol:vandermonde:syst}
\forall r \in \left\{ 1, \cdots, R\right\}, \ \ \bold{w}_r = (-1)^{R-r} \frac{r^{\alpha R}}{\prod_{j=0}^{r-1} (r^{\alpha}-j^{\alpha}) \prod_{j=r+1}^{R} (j^{\alpha} - r^{\alpha})} I_d,
\end{equation}

\noindent where we use the convention $ \prod_{j=R+1}^{R} (j^{\alpha} - r^{\alpha})=1$. Let us note that when $\alpha=1$ this last expression simplifies to $\bold{w}_r = (-1)^{R-r} (r^{R}/(r! (R-r)!)) I_d$, $r=1, \cdots, R$. The first condition in \eqref{twoeq:vandermonde:syst} reads $\sum_{r=1}^{R} \bold{w}_r = I_d$ which implies that $\lim_{n\rightarrow + \infty} \sum_{r=1}^{R} \bold{w}_r \theta^{*,rn} =  \sum_{r=1}^{R} \bold{w}_r  \theta^{*} = \theta^*$. Moreover, multiplying \eqref{first:step:richardromberg} on the left by $\bold{w}^{T}$ yields
\begin{equation}
\label{new:expans:weak:err}
\sum_{r=1}^{R} \bold{w}_r \theta^{*,rn} = \theta^{*} + C_R \frac{1}{n^{\alpha R}} \tilde{\bold{w}}_{R+1} \left(1+ \epsilon_{R+1}(n)\right)
\end{equation}

\noindent where 
\begin{equation}
\label{damp:weight}
\tilde{\bold{w}}_{R+1} = \sum_{r=1}^{R} (-1)^{R-r} \frac{r^{\alpha R}}{\prod_{j=0}^{r-1} (r^{\alpha}-j^{\alpha}) \prod_{j=r+1}^{R} (j^{\alpha} - r^{\alpha})} \frac{1}{r^{\alpha R}} =  \frac{(-1)^{R-1}}{R!^{\alpha}} 
\end{equation}

\noindent and
\begin{equation}
\label{asymp:rest}
\epsilon_{R+1}(n) = \frac{1}{\tilde{\bold{w}}_{R+1}} \sum_{r=1}^{R}  \frac{(-1)^{R-r}}{\prod_{j=0}^{r-1} (r^{\alpha}-j^{\alpha}) \prod_{j=r+1}^{R} (j^{\alpha} - r^{\alpha})}  \epsilon_{r}(n) \rightarrow 0, \ \mbox{ as } \ n\rightarrow + \infty.
\end{equation}

We now approximate the new target $\sum_{r=1}^{R} \bold{w}_{r} \theta^{*,rn}$, by means of $M\in \N^{*}$ steps of $R$ SA schemes which write
\begin{equation}
\label{stoch:approx:algs}
\forall r \in \leftB 1, R\rightB, \ \ \theta^{rn}_{p+1} = \theta^{rn}_{p} - \gamma_{p+1} H(\theta^{rn}_p, (U^{rn})^{p+1}), \ p \in \leftB0, M-1\rightB
\end{equation}

\noindent where $((U^{rn})^{p},r=1,\cdots,R)_{p \in \leftB 1, M\rightB}$ is an i.i.d sequence with the same law as $(U^{rn},r=1,\cdots,R)$, $\theta^{rn}_0$, $r=1\cdots,R$ are the initial conditions independent of the innovation sequence satisfying $\sup_{n\geq1} \E|\theta^{n}_0|^2 < + \infty$ and the sequence $(\gamma_p)_{p\geq1}$ satisfies \eqref{step:rm:assump}. Now the new \emph{statistical error} of the Richardson-Romberg extrapolation estimator writes 
$$
\mathcal{E}^{R-R}_S(n, M) := \sum_{r=1}^{R} \bold{w}_r  (\theta^{*,rn} - \theta^{rn}_{M}).
$$

We are looking for an efficient estimator among the family $\left\{ \sum_{r=1}^R \bold{w}_r \theta^{rn}_{M}, (n, M) \in (\N^{*})^2 \right\}$. To be more precise, we will minimize the computational cost for a given $L^{1}(\P)$-error $\varepsilon >0$. We assume that the cost of a single simulation of $U^{n}$ is proportional to $n$ and is given by $K\times n$, where $K$ is a generic positive constant independent of $n$. It notably corresponds to the case of discretization schemes of a stochastic process. In the case of the Richardson-Romberg method for SA, at each step $p \in \leftB 1, M \rightB$ of the procedure, for every $r\in \leftB1, R\rightB$, one has to simulate the random vector $(U^{n}, U^{2n}, \cdots, U^{Rn})$ so that the global computational cost is given by 
$$
\textnormal{Cost(R-R)} := K M   \sum_{r=1}^{R} r n  = K M n \frac{R(R+1)}{2}.
$$

 Hence the problem of interest writes
$$
(n(\epsilon),M(\epsilon)) = \argmin_{\E|\mathcal{E}^{R-R}_{glob}| \leq \varepsilon}  \textnormal{Cost(R-R)}.
$$

From a practical point of view the constraint: $\E|\mathcal{E}^{R-R}_{glob}| \leq \varepsilon$ is not tractable since one does not have any explicit control on $\E|\mathcal{E}^{R-R}_{glob}|$. Hence one is led to consider some sharp upper bound of this $L^{1}(\P)$-norm, namely
\begin{align}
\label{L1:error:bound}
\E|\mathcal{E}^{R-R}_{glob}| & \leq \left|\sum_{r=1}^{R} \bold{w}_r \theta^{*,rn} - \theta^{*} \right| + \E\left[\left|\sum_{r=1}^{R} \bold{w}_r  (\theta^{*,rn} - \theta^{rn}_{M})\right|\right] \nonumber \\
& \leq  \frac{|C_R|}{(R! n^{R})^{ \alpha}}  \left(1+ |\epsilon_{R+1}(n)|\right) + \E\left[\left|\sum_{r=1}^{R} \bold{w}_r  (\theta^{*,rn} - \theta^{rn}_{M})\right|\right].
\end{align}
 
 Note that the bound \eqref{L1:error:bound} is not tractable since we do not have any closed form expression for the last term appearing in the right-hand side, namely the $L^{1}$-norm (or $L^{2}$-norm) of the statistical error of the Richardson-Romberg SA estimator. Again we will consider some sharp upper bound. In order to derive an explicit control we assume that the following conditions are in force:
\begin{trivlist}
\item[\A{HUI}] $\exists \delta >0$, such that $\forall \theta \in \R^d$, $\sup_{n\in \N^{*}} \E[|H(\theta,U^{n})|^{2+\delta}] < + \infty$.
\item[\A{HC1}] $\exists C>0$ such that $\forall n \in \N^{*}, \forall \theta \in \R^d$, $\E[|H(\theta,U^{n})|^2] \leq C(1+ |\theta-\theta^{*,n}|^2).$
\item[\A{HC2}] $\forall \theta \in \R^d$, $\P(U \notin \mathcal{C}_{\theta}) =0$ with $\mathcal{C}_{\theta}:=\left\{x \in \R^q: x\mapsto H(\theta,x) \mbox{ is continuous at } x \right\}$.

\item[\A{HRG}] There exists $a\in (0,1]$, 
$$
\sup_{n \in \N^{*}, (\theta, \theta') \in (\R^d)^2} \frac{\E|H(\theta, U^{n})-H(\theta', U^{n})|^2}{|\theta-\theta'|^{2a}}< +\infty.
$$
\item[\A{HUA}] For each $n\in \N^{*}$, the map $h^{n}:\theta \in \R^d \mapsto \E[H(\theta,U^{n})]$ is continuously differentiable with $Dh^{n}$ Lipschitz-continuous uniformly in $n$ and there exists $\underline{\lambda}>0 $ s.t. $  \inf_{n \in \N^{*}, \theta \in \R^d} \lambda_{min} \left( (Dh^{n}(\theta) + Dh^{n}(\theta)^{T})/2 \right) > \underline{\lambda}$ where $\lambda_{min}(A)$ denotes the lowest eigenvalue of the matrix $A$.
(\textit{Uniform Attractivity}). 

\item[\A{HS}] The step sequence is given by $\gamma_p = \gamma(p)$, $p\geq1$, where $\gamma$ is a positive function defined on $[0, + \infty[$ decreasing to zero satisfying one of the following assumptions:
\begin{itemize}

\item $\gamma$ varies regularly with exponent $(-\rho)$, $\rho \in (1/2,1)$, that is, for any $x>0$, $\lim_{t\rightarrow + \infty}\gamma(tx)/\gamma(t)=x^{-\rho}$. 

\item for $t\geq1$, $\gamma(t)=\gamma_0/t$ and $\gamma_0$ satisfies $2 \underline{\lambda} \gamma_0>1$.

\end{itemize}
\end{trivlist}

\begin{REM}\label{remark:attractivity} Assumption \A{HUA} already appears in \cite{Duflo1996} and \cite{ben:met:pri}, see also \cite{fri:men:12} and \cite{fat:fri:13} in another context. It allows to control the $L^2$-norm $\E|\theta^{rn}_p-\theta^{*,rn}|^2$, $r\in \leftB1,R\rightB$ with respect to the step $\gamma(p)$ uniformly in $n$, see section \ref{technical:res:sec}, lemma \ref{sstrongerror:tech:lemme} . As discussed in \cite{Kushner2003}, (Chapter 10, Section 5, p.350, Theorem 5.2) if one considers the projected version of the algorithm \eqref{RM} on a bounded convex set $D$, namely
$$
\theta^{n}_{p+1} = \Pi_{D}\left[ \theta^{n}_{p} - \gamma_{p+1} H(\theta^{n}_{p}, (U^{n})^{p+1})\right], \ p \in \leftB 0, M-1\rightB,
$$

\noindent where $\Pi_{D}$ denotes the orthogonal projection operator on $D$ (for instance one may set $D=\Pi_{i=1}^{d} [a_i,b_i]$, $-\infty < a_i < b_i <+\infty$) and  $\forall n\geq1$, $\theta^{*,n} \in int(D)$, as very often happens from a practical point of view, then assumption \A{HUA} can be localized on $D$, that is $ \inf_{n \in \N^{*}, \theta \in D} \lambda_{min} \left( (Dh^{n}(\theta) + Dh^{n}(\theta)^{T})/2 \right) > \underline{\lambda}$.

We also want to point out that if assumption \A{HUA} is satisfied then passing to the limit as $n\rightarrow + \infty$ one easily shows that $ \lambda_{min} \left( (Dh(\theta^*) + Dh(\theta^*)^{T})/2 \right) \geq\underline{\lambda}$. 
\end{REM}

\begin{PROP}{($L^{1}(\P)$ control of the statistical error)}
\label{statistical:error}
Let $R \in \N^{*}$. Suppose that for $r \in \leftB1,R\rightB$, $U^{rn} \overset{\P}{\longrightarrow} U^{r}$ and $\theta^{n}_0 \overset{\P}{\longrightarrow} \theta_0$, as $n\rightarrow +\infty$. Under \A{H-R}, \A{HUI}, \A{HC1}, \A{HC2},  \A{HRG}, \A{HS} and \A{HUA}, one has for some positive constant $C:=C(\gamma,\underline{\lambda})$
\begin{align*}
 \E[|\mathcal{E}^{R-R}_S|] & \leq C \E\left[\left|\sum_{r=1}^{R} \bold{w}_r H(\theta^*,U^{r})\right|^2\right]^{1/2} \gamma^{1/2}(M) \left( 1 + \phi^{R}_1(n) + \phi^{R}_2(M)\right)
\end{align*}

\noindent where $\phi^{R}_1, \phi^{R}_2$ are two positive functions satisfying: $\phi^{R}_1(n)\rightarrow 0$ and $\phi^{R}_2(M)\rightarrow 0$ respectively as $M\rightarrow + \infty$, $n\rightarrow + \infty$ and $\phi^{R}_2$ is non-increasing.
\end{PROP}
\begin{proof}
We define for all $p\geq1$, $\Delta M^{rn}_{p} := h^{rn}(\theta^{rn}_{p-1})-H(\theta^{rn}_{p-1},(U^{rn})^{p})= \E[\left.H(\theta^{rn}_{p-1},(U^{rn})^{p})\right| \mathcal{F}_{p-1}]-H(\theta^{rn}_{p-1},(U^{rn})^{p})$. Recalling that $((U^{n}, U^{2n}, \cdots, U^{rn}, \cdots, U^{Rn})^{p})_{p \in \leftB1,M\rightB}$ is a sequence of i.i.d. random variables we have that $(\Delta M^{rn}_p)_{p\geq1}$, $r\in \leftB1,R\rightB$, are sequences of martingale increments w.r.t. the natural filtration of the stochastic approximation schemes $\mathcal{F}:=(\mathcal{F}_p:=\sigma(\theta^{rn}_0, (U^{rn})^{1}, \cdots, (U^{rn})^{p}, r=1, \cdots, R); p \geq1)$. Using Taylor's formula we get for $p\geq 0$ and $r \in \leftB1, R\rightB$
\begin{align*}
\theta^{rn}_{p+1} - \theta^{*,rn} & = \theta^{rn}_p - \theta^{*,rn} - \gamma_{p+1} h^{rn}(\theta^{rn}_p) + \gamma_{p+1} \Delta M^{rn}_{p+1} \\
& = \theta^{rn}_p - \theta^{*,rn} - \gamma_{p+1} Dh(\theta^{*}) (\theta^{rn}_p-\theta^{*,rn})  \\
&  + \gamma_{p+1} \left( Dh(\theta^*) - \int_0^1 d\lambda Dh^{rn}(\theta^{*,rn}+(1-\lambda)(\theta^{rn}_p-\theta^{*,rn}))\right) (\theta^{rn}_p-\theta^{*,rn}) + \gamma_{p+1} \Delta M^{rn}_{p+1}.
\end{align*}

Hence by a simple induction argument one has for $(r,M) \in \leftB1,R\rightB \times \N^{*}$
\begin{align}
\label{first:rec}
\theta^{rn}_{M} - \theta^{*,rn} & = \Pi_{1,M} (\theta^{rn}_0 - \theta^{*,rn}) + \sum_{k=1}^{M} \gamma_k \Pi_{k+1,M} \Delta M^{rn}_{k} + 
\sum_{k=1}^{M} \gamma_k \Pi_{k+1,M} R^{rn}_{k-1} 
\end{align}

\noindent where $R^{rn}_k = \left( Dh(\theta^*) - \int_0^1 d\lambda Dh^{rn}(\theta^{*,rn}+(1-\lambda)(\theta^{rn}_k-\theta^{*,rn}))\right) (\theta^{rn}_k-\theta^{*,rn})$ and $\Pi_{k,M}:= \prod_{j=k}^M (I_d-\gamma_j Dh(\theta^*))$, with the convention that $\Pi_{M+1,M}=I_d$. Multiplying \eqref{first:rec} on the left by $\bold{w}_r$ given by \eqref{sol:vandermonde:syst} and summing w.r.t $r$ lead to
\begin{align}
\label{sec:rec}
-\mathcal{E}^{R-R}_S & = \Pi_{1,M} \left(\sum_{r=1}^{R} \bold{w}_r(\theta^{rn}_0 - \theta^{*,rn}) \right) + \sum_{k=1}^{M} \gamma_k \Pi_{k+1,M} \left(\sum_{r=1}^{R}\bold{w}_r \Delta M^{rn}_{k}\right) + 
\sum_{k=1}^{M} \gamma_k \Pi_{k+1,M} \left(\sum_{r=1}^{R} \bold{w}_r R^{rn}_{k-1} \right)
\end{align}

Ought to the Minkowski inequality it is sufficient to bound the $L^{1}(\P)$-norm of each term in the above decomposition. First, since $-Dh(\theta^*)$ is a Hurwitz matrix, $\forall \lambda \in [0, \underline{\lambda})$, there exists $C>0$ such that for any $k\leq n$, $\| \Pi_{k,n} \| \leq C \prod_{j=k}^{n} (1-\lambda \gamma_j) \leq C \exp(-\lambda \sum_{j=k}^n \gamma_j)$. We refer to \cite{Duflo1996} and \cite{ben:met:pri} for more details. Hence, one has for all $\eta \in (0,\underline{\lambda})$
$$
\E[|\Pi_{1,M} (\sum_{r=1}^{R} \bold{w}_r(\theta^{rn}_0 - \theta^{*,rn}) ) |] \leq ||\Pi_{1,M}|| \E[| \sum_{r=1}^{R} \bold{w}_r(\theta^{rn}_0 - \theta^{*,rn}) |] \leq C e^{-(\underline{\lambda}-\eta)\sum_{k=1}^M \gamma_k} \E[| \sum_{r=1}^{R} \bold{w}_r(\theta^{rn}_0 - \theta^{*,rn}) |].
$$

\noindent where $||.||$ stands for the matrix norm on $\R^d \otimes \R^d$. For the second term, recalling that $\sum_{r=1}^{R}\bold{w}_r \Delta M^{rn}_{k}$ is a martingale increment, one has
\begin{align}
\label{mart:term:statist:err}
\E\left[ \left| \sum_{k=1}^{M} \gamma_k \Pi_{k+1,M} (\sum_{r=1}^{R}\bold{w}_r \Delta M^{rn}_{k})\right|^2\right]^{1/2} & \leq \left( \sum_{k=1}^{M} \gamma^{2}_k ||\Pi_{k+1,M}||^2 \E\left[ \left|\sum_{r=1}^{R}\bold{w}_r \Delta M^{rn}_{k}\right|^2\right]\right)^{1/2} .
\end{align}

Similarly for the last term, one has
\begin{align}
\label{remainder:term}
\E\left[ \left| \sum_{k=1}^{M} \gamma_k \Pi_{k+1,M} \left(\sum_{r=1}^{R} \bold{w}_r R^{rn}_{k-1} \right) \right| \right] & \leq \sum_{k=1}^M \gamma_k ||\Pi_{k+1,M}|| \E\left| \sum_{r=1}^{R} \bold{w}_r R^{rn}_{k-1} \right|.
\end{align}

We now study the limit of each bound as $n$ and $M$ go to infinity. For the first term, observe that $\sum_{r=1}^{R} \bold{w}_r (\theta^{rn}_0-\theta^{*,rn}) \overset{\P}{\longrightarrow} \sum_{r=1}^{R} \bold{w}_r (\theta_0-\theta^{*}) =  \theta_0-\theta^{*}$ as $n\rightarrow + \infty$. Moreover, since $\sup_{n\geq1}\E|\theta^{n}_0|^{2} < + \infty$, by uniform integrability one has $\E|\sum_{r=1}^{R} \bold{w}_r (\theta^{rn}_0-\theta^{*,rn})| \rightarrow \E|\theta_0-\theta^{*}|$ as $n\rightarrow + \infty$. If $\gamma(p)=\gamma_0/p$ we select $\eta$ such that $2(\underline{\lambda}-\eta)\gamma_0 >1$ otherwise we set $\eta< \underline{\lambda}$ which implies that $\exp(-(\underline{\lambda}-\eta) \sum_{j=k}^n \gamma_j) = \gamma^{1/2}(M) \phi^{R}_2(M)$ with $\phi^{R}_2(M)\rightarrow 0$ as $M\rightarrow +\infty$. Hence we get
$$
\E[|\Pi_{1,M} (\sum_{r=1}^{R} \bold{w}_r(\theta^{rn}_0 - \theta^{*,rn}) ) |] \leq C \gamma^{1/2}(M) \phi^{R}_2(M).
$$

\noindent Let us now study the second term. Define for $k\geq1$, $\Delta N^{rn}_k = h^{rn}(\theta^{*}) - H(\theta^{*},(U^{rn})^{k})$ then by the Cauchy-Schwarz inequality and \A{HRG} one has 
\begin{align*}
\left|\E\left[ \left|\sum_{r=1}^{R}\bold{w}_r \Delta M^{rn}_{k}\right|^2\right] - \E\left[ \left|\sum_{r=1}^{R}\bold{w}_r \Delta N^{rn}_{k}\right|^2\right] \right| & \leq C_R \left(\sum_{r=1}^{R}||\bold{w}_r|| \E\left[\left|  \Delta M^{rn}_{k} - \Delta N^{rn}_{k} \right|^{2}\right] \right) ^{1/2} \\
& \times \left(\E[|H(\theta^{rn}_{k-1}, (U^{rn})^k)|^2]^{1/2} + \E[|H(\theta^*, (U^{rn})^k)|^2]^{1/2} \right) \\
& \leq C_R \max_{1\leq r\leq R}\E[|\theta^{rn}_{k-1} - \theta^{*}|^{2a}]^{1/2} \\
& \leq C_R (\gamma^{a/2}_{k} +  n^{-a \alpha} )
\end{align*}

\noindent where we used lemma \ref{sstrongerror:tech:lemme} and $\max_{1\leq r \leq R} |\theta^{*,rn} - \theta^{*}| \leq C n^{-\alpha}$ for the last inequality. Now observe that $\E\left[\left| \sum_{r=1}^R \bold{w}_r \Delta N^{rn}_k \right|^{2}\right] = \E\left[\left| \sum_{r=1}^R \bold{w}_r (h^{rn}(\theta^{*}) - H(\theta^{*}, U^{rn}) ) \right|^{2}\right]$ so  that using \A{HC2} and $U^{rn} \overset{\P}{\longrightarrow} U^{r}$ as $n\rightarrow + \infty$, one has $\sum_{r=1}^R \bold{w}_r (h^{rn}(\theta^{*}) - H(\theta^{*}, U^{rn}) )  \overset{\P}{\longrightarrow} -\sum_{r=1}^R \bold{w}_r  H(\theta^{*}, U^{r})$ as $n\rightarrow + \infty$. From \A{HUI} we deduce the $L^{2}$-uniform integrability of the family $\left\{\sum_{r=1}^R \bold{w}_r (h^{rn}(\theta^{*}) - H(\theta^{*}, U^{rn}) ), n\geq1\right\}$ which yields 
$$
\E\left[\left| \sum_{r=1}^R \bold{w}_r \Delta N^{rn}_k \right|^{2}\right] \longrightarrow \E\left[\left|\sum_{r=1}^R \bold{w}_r  H(\theta^{*}, U^{r})\right|^2\right], \ \ n\rightarrow + \infty.
$$

Plugging the above estimates into \eqref{mart:term:statist:err}, we derive the following bound
\begin{align}
\E\left[ \left| \sum_{k=1}^{M} \gamma_k \Pi_{k+1,M} (\sum_{r=1}^{R}\bold{w}_r \Delta M^{rn}_{k})\right|^2\right]^{1/2} & \leq \E\left[\left|\sum_{r=1}^R \bold{w}_r  H(\theta^{*}, U^{r})\right|^2\right]^{1/2}  \left(\sum_{k=1}^{M} \gamma^{2}_k ||\Pi_{k+1,M}||^2 \right)^{1/2} (1+ \phi^{R}_1(n)) \\
& + C_R \left(\sum_{k=1}^{M} \gamma^{2}_k \gamma^{a/2}_k ||\Pi_{k+1,M}||^2 \right)^{1/2},  \nonumber
\end{align}

\noindent with $\phi^{R}_1(n) \rightarrow 0$ as $n\rightarrow + \infty$. Using lemma \ref{stepseq:tech:lemme}, we successively derive that $\left(\sum_{k=1}^{M} \gamma^{2}_k ||\Pi_{k+1,M}||^2 \right)^{1/2} \leq C \gamma^{1/2}(M)$ for some positive constant $C(\gamma,\underline{\lambda})$ and $ \left(\sum_{k=1}^{M} \gamma^{2}_k \gamma^{a/2}_k ||\Pi_{k+1,M}||^2 \right)^{1/2}=o(\gamma^{1/2}(M))=\gamma^{1/2}(M) \phi^{R}_2 (M)$ as $M\rightarrow + \infty$. We now focus on the last term. Let us first observe that using \A{H-R} and since $Dh^{rn}$ is Lipschitz (uniformly in $n$) one has
\begin{align*}
\left| R^{rn}_k  \right| & = \left| \left( Dh(\theta^*) -Dh^{rn}(\theta^*) + \int_0^1 d\lambda \left(Dh^{rn}(\theta^*) - Dh^{rn}(\theta^{*,rn}+(1-\lambda)(\theta^{rn}_k-\theta^{*,rn}))\right) \right)(\theta^{rn}_k-\theta^{*,rn}) \right| \\
& \leq C \left( \max_{1 \leq r \leq R} ||Dh(\theta^*) -Dh^{rn}(\theta^*)|| + |\theta^{rn}_k - \theta^{*,rn}| \right) |\theta^{rn}_k - \theta^{*,rn}|
\end{align*}

\noindent so that plugging this estimate in \eqref{remainder:term} and using lemma \ref{sstrongerror:tech:lemme} lead to 
$$
\E\left[ \left| \sum_{k=1}^{M} \gamma_k \Pi_{k+1,M} \left(\sum_{r=1}^{R} \bold{w}_r R^{rn}_{k-1} \right) \right| \right]  \leq C \left( \sum_{k=1}^M (\gamma^{3/2}_k \max_{1 \leq r \leq R} ||Dh(\theta^*) -Dh^{rn}(\theta^*)|| + \gamma^2_k)  ||\Pi_{k+1,M}||  \right).
$$

Finally lemma \ref{stepseq:tech:lemme} and since $\max_{1 \leq r \leq R} ||Dh(\theta^*) -Dh^{rn}(\theta^*)|| \rightarrow 0$ as $n\rightarrow + \infty$ also imply 
$$
\max_{1 \leq r \leq R} ||Dh(\theta^*) -Dh^{rn}(\theta^*)||  \left(\sum_{k=1}^M \gamma^{3/2}_k  ||\Pi_{k+1,M}||\right) \leq C \gamma^{1/2}(M)\phi^{R}_1(n)
$$

\noindent and applying again Lemma \ref{stepseq:tech:lemme} with $a=1/2$ and $v_k=\gamma^{1/2}_k$, one has:
$$
\sum_{k=1}^M  \gamma^2_k  ||\Pi_{k+1,M}|| = o(\gamma^{1/2}(M)) = \gamma^{1/2}(M) \phi^{R}_2(M).
$$
\end{proof}

From the previous computations we are naturally led to consider the following suboptimal computational cost optimization problem
\begin{equation}
\label{subopt:problem}
(n(\epsilon),M(\epsilon)) = \argmin_{\mu_R n^{-\alpha R} \left(1+ |\epsilon_{R+1}(n)|\right) + \nu_R \gamma^{1/2}(M) (1+\phi^{R}_1(n)+\phi^{R}_2(M)) \leq  \varepsilon} \textnormal{Cost(R-R)}
\end{equation}

\noindent where $\mu_R = \frac{|C_R|}{R!^{ \alpha}}$ and $\nu_R = C \E\left[\left|\sum_{r=1}^{R} \bold{w}_r H(\theta^*,U^{r})\right|^2\right]^{1/2}$.

\begin{PROP}(Computational cost optimization)
\label{sub:opt:res}
Let $R\in \N^{*}$. Suppose that the assumptions of Proposition \ref{statistical:error} are satisfied. Suppose that the step sequence $\gamma$ is given by: $\gamma(p)=\gamma_0/p^{\beta}$, $\gamma_0>0$, $p>0$, $\beta \in (1/2,1]$. The multi-step Richardson-Romberg SA estimator of order $R$ satisfies
$$
\inf_{\mu_R n^{-\alpha R} \left(1+ |\epsilon_{R+1}(n)|\right) + \nu_R \gamma^{1/2}(M) (1+\phi^{R}_1(n)+\phi^{R}_2(M)) \leq  \varepsilon} \textnormal{Cost(R-R)} \sim  K \frac{R(R+1)}{2} \gamma^{\frac{1}{\beta}}_0 \nu^{\frac{2}{\beta}}_R \mu^{\frac{1}{\alpha R}}_R \frac{1}{\varepsilon^{\frac{2}{\beta}+\frac{1}{\alpha R}}} \left( 1+ \frac{2\alpha R}{\beta}\right)^{\frac{1}{\alpha R}} \left(1+ \frac{\beta}{2\alpha R}\right)^{\frac{2}{\beta}} 
$$

\noindent as $\varepsilon \rightarrow 0$. Eventually this asymptotically optimal bound may be achieved with parameters satisfying:
\begin{equation}
\label{asympt:opt:param}
n(\varepsilon ) \sim \left(\frac{2 \alpha R}{\beta}  + 1 \right)^{\frac{1}{\alpha R}}  \mu^{\frac{1}{\alpha R}}_R \varepsilon^{-\frac{1}{\alpha R}} \ \ \mbox{and} \ \ M(\varepsilon) \sim  \gamma^{\frac{1}{\beta}}_0 \nu^{\frac{2}{\beta}}_R \left(1+ \frac{\beta}{2\alpha R}\right)^{\frac{2}{\beta}} \varepsilon^{-\frac{2}{\beta}} \ \ \mbox{as} \ \varepsilon \rightarrow 0.
\end{equation}

\end{PROP}

\begin{proof} Let us note that the cost minimization problem \eqref{subopt:problem} is lower-bounded by the more tractable problem
\begin{equation}
\label{lower:bound}
\inf_{\mu_R n^{-\alpha R}  + \nu_R \gamma^{1/2}(M) \leq  \varepsilon} \textnormal{Cost(R-R)} = \inf_{\mu_R n^{-\alpha R}  <  \varepsilon} K \gamma^{-1}\left( \frac{(\epsilon-\mu_R  n^{-\alpha R})^2}{\nu^2_R }  \right)n \frac{R(R+1)}{2}
\end{equation}

\noindent with $M = \gamma^{-1}\left( \frac{(\varepsilon - \mu_R n^{-\alpha R})^2}{\nu^2_R }\right)= \gamma^{1/\beta}_0 \nu^{2/ \beta}_R  (\varepsilon - \mu_R  n^{-\alpha R})^{-2/\beta}$. This optimization problem can be solved explicitly, more precisely the optimal parameters are given by
$$
n(\varepsilon) = \left(\frac{2 \alpha R}{\beta}+1  \right)^{\frac{1}{\alpha R}} \mu^{\frac{1}{\alpha R}}_R\varepsilon^{-\frac{1}{\alpha R}} , \ \ M(\varepsilon) = \gamma^{\frac{1}{\beta}}_0 \nu^{\frac{2}{ \beta}}_R \left(1+ \frac{\beta}{2\alpha R}\right)^{\frac{2}{\beta}} \varepsilon^{-\frac{2}{\beta}}.
$$

The "liminf" side of the result clearly follows by plugging this solution into \eqref{lower:bound}. Now set 
$$
n(\varepsilon) = \left(\frac{2 \alpha R}{\beta}+1  \right)^{\frac{1}{\alpha R}} \mu^{\frac{1}{\alpha R}}_R\varepsilon^{-\frac{1}{\alpha R}} , \ \ M(\varepsilon) =  \gamma^{-1}\left( \frac{(\varepsilon - \mu_R (1+|\varepsilon_{R+1}(n(\varepsilon))|) n^{-\alpha R}(\varepsilon))^2}{\nu^2_R \left(1+\frac{\beta}{2\alpha R}\right)^2 \left(1+\phi^{R}_1(n(\varepsilon))+\phi^{R}_2(\gamma^{\frac{1}{\beta}}_0 \nu^{\frac{2}{ \beta}}_R \left(1+ \frac{\beta}{2\alpha R}\right)^{\frac{2}{\beta}} \varepsilon^{-\frac{2}{\beta}})\right)^2}\right). 
$$

Since $\phi^{R}_2$ is non-increasing, the couple $(n(\varepsilon),M(\varepsilon))$ satisfies the constraint $\mu_R n^{-\alpha R} \left(1+ |\epsilon_{R+1}(n)|\right) + \nu_R \gamma^{1/2}(M) (1+\phi^{R}_1(n)+\phi^{R}_2(M)) \leq  \varepsilon$ so that  the cost minimization problem \eqref{subopt:problem} is upper-bounded by 
\begin{align*}
 K \frac{R(R+1)}{2}  \mu^{\frac{1}{\alpha R}}_R \varepsilon^{-\frac{1}{\alpha R}}& \left( 1+ \frac{2\alpha R}{\beta}\right)^{\frac{1}{\alpha R}} \gamma^{\frac{1}{\beta}}_0 \nu^{\frac{2}{\beta}}_R \left(1+ \frac{\beta}{2\alpha R}\right)^{\frac{2}{\beta}} \varepsilon^{-\frac{2}{\beta}} \left( 1 -  (1+|\varepsilon_{R+1}(n(\varepsilon))|) \frac{\beta}{2\alpha R +\beta}\right)^{\frac{2}{\beta}}  \\
 & \times \left(1+\phi^{R}_1(n(\varepsilon))+\phi^{2}_R(\gamma^{\frac{1}{\beta}}_0 \nu^{\frac{2}{ \beta}}_R \left(1+ \frac{\beta}{2\alpha R}\right)^{\frac{2}{\beta}} \varepsilon^{-\frac{2}{\beta}})\right)^{\frac{2}{\beta}}
\end{align*}

and the result follows by letting $\varepsilon$ goes to zero. 

 \end{proof}

\begin{REM}\textbf{(Choice of the step sequence)}
According to Proposition \ref{sub:opt:res}, it is optimal to set $\beta=1$ to achieve a minimal asymptotic complexity. In this case a constraint appear on $\gamma_0$: $2 \underline{\lambda} \gamma_0>1$. Let us note that for $\beta=1$ a simple computation shows that the constant $C$ appearing in $\nu_R$ is equal to $\gamma_0/(2 \underline{\lambda} \gamma_0 -1)^{1/2}$ which reaches its minimum (as a function of $\gamma_0$) at $\gamma_0= 1/\underline{\lambda}$. However the main drawback with this choice is that the constant $\underline{\lambda}$ is not known to the experimenter so that one is led to make a blind choice in practical implementation. 
\end{REM}

\begin{REM}\textbf{(Control of the variance)}\label{control:var:rem}
Let us note that when one decides to implement the Richardson-Romberg extrapolation SA scheme with an innovation satisfying $U^{r}= U$ a.s. $r=1, \cdots, R$ then one has $H(\theta^*,U^{r}) = H(\theta^*, U)$ $a.s.$ for every $r \in \leftB 1, R\rightB$ so that using \eqref{twoeq:vandermonde:syst} yields
$$
\E\left[\left| \sum_{r=1}^R \bold{w}_r H(\theta^*,U^{r})\right|^2\right] = \E\left[\left| (\sum_{r=1}^R \bold{w}_r) H(\theta^*, U)\right|^2\right] = \E\left[\left| H(\theta^*, U)\right|^2\right].
$$

Hence we clearly see that this choice leads to a control in the $L^{1}$-norm of the statistical error of the multi-step Richardson-Romberg SA estimator. On the opposite considering mutually independent innovations $U^{r}$ lead to an explosion of the previous control with respect to $R$. Indeed one has 
\begin{align*}
\E\left[\left| \sum_{r=1}^R \bold{w}_r H(\theta^*,U^{r})\right|^2\right] & = \left(\sum_{r=1}^{R}  \frac{r^{2\alpha R}}{\prod_{j=0}^{r-1} (r^{\alpha}-j^{\alpha})^2 \prod_{j=r+1}^{R}  (j^{\alpha} - r^{\alpha})^2} \right) \E\left[\left| H(\theta^*, U )\right|^2\right] \\
& \geq \left(\frac{R^{ R}}{R!}\right)^{2\alpha} \E\left[\left| H(\theta^*, U)\right|^2\right] \\
& \sim \left( \frac{e^{R}}{\sqrt{2 \pi}\sqrt{R}} \right)^{2\alpha} \E\left[\left| H(\theta^*, U)\right|^2\right] \ \ \mbox{as} \ R\rightarrow + \infty,
\end{align*}

\noindent where we used \eqref{sol:vandermonde:syst} for the first equality.

For instance when one is concerned with the discretization of a Brownian diffusion, the first aforementioned case consists in implementing the Richardson-Romberg method with $R$ Euler schemes devised with the \emph{same Brownian motion} $W$ namely $W^{r} = W, \ \ r=1, \cdots, R$ whereas the second case consists in implementing the method with mutually independent Brownian motions $W^{r}$. The optimality of this choice is discussed in \cite{GPag:07}.

\end{REM}

\subsection{Comparison with the crude stochastic approximation estimator}
Under the assumptions of Proposition \ref{statistical:error} with $R=1$, the global error for the crude SA estimator satisfies
\begin{align*}
\E \left[ \left| \mathcal{E}_{glob}(M, \gamma, H) \right| \right] = \E\left[ \left| \theta^{*} - \theta^{*,n} + \theta^{*,n}- \theta^{n}_M \right|  \right] &  \leq \frac{|C_1|}{n^{\alpha}} (1+ |\varepsilon_1(n)|) +  \E\left[ \left|\theta^{*,n}- \theta^{n}_M \right| \right] \\
& \leq  \frac{|C_1|}{n^{\alpha}} \left(1+ |\epsilon_{1}(n)|\right) + C \E\left[\left| H(\theta^*,U)\right|^2\right]^{\frac12} \gamma^{\frac12 }(M) (1+\phi_1(n)+\phi_2(M)),
\end{align*}

\noindent with a computational cost given by $\textnormal{Cost(C-S)} := K M n $. Hence a similar result as in Proposition \ref{sub:opt:res} holds.
\begin{PROP}
\label{optim:prob:crude}
Assume that the assumptions of Proposition \ref{statistical:error} with $R=1$ hold. Suppose that the step sequence $\gamma$ is given by: $\gamma(p)=\gamma_0/p^{\beta}$, $\gamma_0>0$, $p>0$, $\beta \in (1/2,1]$. The crude SA estimator satisfies
$$
\inf_{ |C_1| n^{-\alpha } \left(1+ |\epsilon_{1}(n)|\right) + \nu_1  \gamma^{1/2}(M) (1+\phi_1(n)+\phi_2(M)) \leq  \varepsilon} \textnormal{Cost(C-S)} \sim K \gamma^{\frac{1}{\beta}}_0 \nu^{\frac{2}{\beta}}_1 |C_1|^{\frac{1}{\alpha }} \frac{1}{\varepsilon^{\frac{2}{\beta}+\frac{1}{\alpha }}} \left( 1+ \frac{2\alpha}{\beta}\right)^{\frac{1}{\alpha}} \left(1+ \frac{\beta}{2\alpha}\right)^{\frac{2}{\beta}} 
$$

\noindent as $\varepsilon \rightarrow 0$ with $\nu_1 = C \E\left[\left| H(\theta^*,U)\right|^2\right]^{\frac12} $. Eventually this asymptotically optimal bound may be achieved with parameters satisfying:
\begin{equation}
\label{asympt:opt:param:crude}
n(\varepsilon ) \sim \left(\frac{2 \alpha}{\beta}  + 1 \right)^{\frac{1}{\alpha}}  |C_1|^{\frac{1}{\alpha}} \varepsilon^{-\frac{1}{\alpha}} \ \ \mbox{and} \ \ M(\varepsilon) \sim  \gamma^{\frac{1}{\beta}}_0 \nu^{\frac{2}{\beta}}_1 \left(1+ \frac{\beta}{2\alpha }\right)^{\frac{2}{\beta}} \varepsilon^{-\frac{2}{\beta}} \ \ \mbox{as} \ \varepsilon \rightarrow 0.
\end{equation}
\end{PROP}
\section{Application: Estimation of the quantile of a component of a SDE}\label{applic:sec}
In this section, we show how the previous results can be applied to the estimation of the quantile of a stochastic process solution to a stochastic differential equation.
Also, when the exact value of a constant is not important we may repeat the same symbol for constants that may change from one line to next. 

\subsection{Notations and Hypotheses.}
Let $(\Omega,\F,(\F_t)_{t\ge 0},\P) $ be a filtered probability space satisfying the usual conditions and $(Z_t)_{t\ge 0} $ be a $d$-dimensional $(\F_t)_{t\ge 0} $ symmetric $\alpha$-stable process, for $\alpha\in (0,2]$, that is a c\`adl\`ag process with independent and stationary increments with the scaling property $Z_{ct} \overset{(d)}{=} c^{1/\alpha}Z_{t}$. 
Note that the case $\alpha = 2$ corresponds to the standard Brownian motion. It is also the only case where $Z$ is a continuous process.
When $\alpha<2$, the Stable process is discontinuous and its L\'evy-Khintchine exponent writes for all $p\in\R^d $,
$$
\E\left( e^{i\langle p,Z_t \rangle}\right)= \exp\left( -t \int_{S^{d-1}} |\langle p,\vartheta\rangle|^\alpha \mu(d\vartheta)\right).
$$
We refer to the measure $\mu$ as the \textit{spectral measure} of $Z$.
It is related to the L\'evy measure of the process $Z$ as follows.
Denote $\nu$ the L\'evy measure of $Z$, $\nu$ factorizes in $\nu(dz) = C_\alpha \frac{d|z|}{|z|^{1+\alpha}} \mu(\bar{z})$, where $z=(|z|,\bar{z}) \in \R_+ \times S^{d-1}$ stands for the polar coordinates.
For the exact value of $C_\alpha$, we refer to Sato \cite{sato}.
Let us  consider a $d$-dimensional process $(X_t)_{t\ge 0}=(X_t^1,\dots,X_t^d)_{t\ge 0} $ with dynamics:
\begin{equation}
\label{EqDfSt}
X_t=x+\int_0^t b(X_{s-})ds+\int_0^t \sigma(X_{s-})dZ_s,
\end{equation}

\noindent where $b: \R^d \rightarrow \R^d$ and $\sigma: \R^d \rightarrow \R^d \otimes \R^d$. We fix the time horizon $T=1$. Let us denote by $\P_x$ (resp. $\P_{t,x}$, $t \in (0,1]$) the conditional probability given $\left\{X_0 = x \right\}$ (resp. $\left\{X_t = x \right\}$). For a given level $\ell \in (0,1)$, we are interested in the computation of the quantile at level $\ell$ of the random variable $X^{d}_1$ defined as:
$$
\theta^* = \inf \{ \theta \in \R : \P_x(X_1^d \le \theta) \ge \ell \}.
$$

%
Since $\lim_{\theta \rightarrow + \infty} \P_x(X^{d}_1 \leq \theta) = 1$, we have $\{ \theta \in \R : \P_x(X_1^d \le \theta) \ge \ell \} \neq  \emptyset$. Moreover, we have $\lim_{\theta \rightarrow - \infty} \P_x(X^{d}_1 \leq \theta) = 0$, which implies that $\{ \theta \in \R : \P_x(X_1^d \le \theta) \ge \ell \} $ is bounded from below so that $\theta^*$ always exists. Assuming that the distribution of $X^{d}_1$ has no atoms, the quantile at level $\ell$ is the lowest solution of the equation:
$$
\P_x(X^{d}_1 \leq \theta ) = \ell.
$$

If the distribution function is (strictly) increasing, which is notably the case if the process $X$ solution of \eqref{EqDfSt} admits a positive density $p(1,x,.)$, the solution to the above equation is unique, otherwise, there may be more than one solution. Now since the law of $X^{d}_1$ is not known explicitly, the quantile $\theta^*$ cannot be computed and one has to approximate the dynamics by a discretization scheme that can be simulated. Let us note that the estimation of the quantile of a component of a Brownian diffusion process has already been investigated in \cite{Talay2004}. For a given time step $\Delta = \frac{1}{n}$, $n \in \N^{*}$, setting for all $i \in \N$, $t_i = i \Delta$, we consider the standard Euler scheme defined as follows:
\begin{equation}\label{EULER_EDS}
X_t^n =x+ \int_0^t b(X_{\phi(s)}^n) ds  + \int_0^t \sigma(X_{\phi(s)}^n) dZ_s \ \ \phi(s) = \sup\left\{t_i: \ t_i \leq s\right\}.
 \end{equation}

Then one approximates $\theta^*$ by $\theta^{*,n}$ the quantile at level $\ell$ of $X^{n,d}_1$.  We denote by \textbf{[A]} the following set of assumptions.
Fix an integer $m\in \N$ which will hereafter refer to the regularity of the coefficients. 
\begin{itemize}
\item[\textbf{[A-1]}] $b \in \mathcal C^m(\R^d,\R^d)$ and $\sigma \in \mathcal C^m(\R^d,\R^d\otimes \R^d)$ with bounded derivatives.
Also, when $\alpha \le 1$, we put $b=0$.
\item[\textbf{[A-2]}] When $\alpha<2$ for all $x,\xi \in \R^d$, there exists $C>1$ such that:
\begin{equation*}
C^{-1} |\xi|^2\le \langle \xi , \sigma(x) \xi \rangle \le C|\xi|^2.
\end{equation*}

\noindent When $\alpha=2$, setting $\Sigma(x)=\sigma(x)\sigma(x)^{T}$, for all $x,\xi \in \R^d$, there exists $C>1$ such that:
\begin{equation*}
C^{-1} |\xi|^2\le \langle \xi , \Sigma(x) \xi \rangle \le C|\xi|^2.
\end{equation*}

\item[\textbf{[A-3]}] When $\alpha <2$, the spectral measure $\mu$ has a $\mathcal C^m(S^{d-1})$ surface density and satisfies:
for all $\xi\in \R^q$, there exists $C>1$ such that:
\begin{equation}\label{ND_spectral_measure}
C^{-1} |\xi|^\alpha \le \int_{S^{d-1}}|\langle \xi, \vartheta \rangle |^\alpha \mu(d\vartheta) \le C |\xi|^\alpha.
\end{equation}
\end{itemize}

\begin{PROP}\label{exist:dens:prop}
Assume that $\alpha \in (0,2]$ and that \textbf{[A]} is in force. For every $t>0$, the solutions $X_t$, $X_t^n$, of the SDE \eqref{EqDfSt} and \eqref{EULER_EDS} have a strictly positive densities with respect to the Lebesgue measure.
Consequently, the quantile is uniquely defined.
Moreover, those densities are in $\mathcal C^m(\R^d,\R^d)$ if $\alpha > 1$, and in $\mathcal C^{m-1}(\R^d,\R^d)$ when $\alpha \le 1$.
\end{PROP}
 
 We refer to the work of Kolokoltsov \cite{kolo:00} for the proof in the Stable case, who also derived Aronson's estimates with time singularity depending on the index $\alpha$. In the Brownian case, i.e. $\alpha = 2 $, if the drift $b$ is a measurable bounded function and the diffusion coefficient $\sigma$ is $\eta$-H\"older continuous, $\eta>0$, and satisfies \textbf{[A-2]} then the aforementioned densities exists, are positive and satisfy Gaussian Aronson's estimates (see e.g. \cite{friedman:64} and \cite{lem:men:2010} for the density of the Euler scheme).
 
\begin{PROP}\label{conv:quantile:prop} For $\alpha \in (0,2)$ assume that \textbf{[A]} for $m\ge2$. For $\alpha = 2$, assume the drift $b$ and the diffusion coefficient $\sigma$ are Lipschitz-continuous bounded functions and that $\sigma$ satisfies \textbf{[A-2]}.Then one has
$$
\theta^{*,n} \rightarrow \theta^*, \ n\rightarrow +\infty.
$$
\end{PROP}

\begin{proof} Let $n\in \N^{*}$ and denote by $F, \ F^{n}$ the distribution function of $X^{d}_1$ and $X^{n,d}_1$ respectively. Since $b$ and $\sigma$ are Lipschitz we know that $(X^{n,d}_1)_{n\geq1}$ converges in distribution to $X^{d}_1$. Moreover, the function $F$ is continuous so that $(F^{n})_{n\geq1}$ converges uniformly to $F$. Hence, we conclude that $F(\theta^{*,n}) \rightarrow \ell$, $n\rightarrow + \infty$. Now remark that from Proposition \ref{exist:dens:prop} since $X^{d}_1$ has a strictly positive density the function $F$ is one-to-one which in turn implies that $F^{-1}$ exists and is continuous so that $\theta^{*,n} \rightarrow F^{-1}(\ell) = \theta^*$.
\end{proof}
 
%

From Proposition \ref{exist:dens:prop} (existence of a positive density for $X^{n,d}_1$) the quantile $\theta^{*,n}$ at level $\ell$ of the random variable $X^{n,d}_1$ is the unique solution of the equation
$$
\P_x\left( X^{n,d}_1 \leq \theta \right) = \ell.
$$

In this section, we are interested in giving an expansion for the error $\theta^*-\theta^{*,n}$ in powers of $n^{-1}$, using Theorem \ref{implic:discret:err:main_result}. Actually, we will prove that \textbf{[A]} implies \textbf{[H-k]}, for a desired $k>0$.
As we can see, Theorem \ref{implic:discret:err:main_result} requires an expansion of $h^n-h$ and its derivatives up to order $k>0$ in order to have an expansion of $\theta^{*}-\theta^{*,n}$ at the same order.
Regularity of the function $h$ may be obtained mainly by two means: either the function $H$ is smooth w.r.t. the variable $\theta$ (with polynomial growth w.r.t $\theta$ and $x$) or the laws of $X_T$ and $X_T^n$ are smooth. Concerning the expansion of the difference $\partial^{k}_\theta h - \partial^{k}_\theta h^{n}$ it may also be obtained by two means: in the regular setting i.e. when the function $x \mapsto \partial^{k}_{\theta} H(\theta, x)$ and the coefficients $b$ and $\sigma$ are regular (say $b, \ \sigma, \ \partial^{k}_{\theta} H(\theta, .)$ are $\mathcal{C}^{R+5}_b$) one may use standard tools such as the one developed in Talay-Tubaro \cite{tala:tuba:90} (in the Brownian case); or in the (Hypo-)elliptic setting, the laws of $X_T$ and $X_T^n$ are smooth. Here, we are in the latter case. Indeed, the estimation of the quantile of a diffusion can be seen as an inverse problem, by setting $H(\theta,x) = 1-\frac{1}{1-\ell}\textbf{1}_{ \{x^d\ge \theta \} } $. We thus see that regularity of $H$ fails.
However, for $\theta \in \R$, we have:
$$
h(\theta) - h^n(\theta) = \frac{1}{1-\ell}\left(\mathbb{P}^x (X_1^d \le \theta)  - \mathbb{P}^x (X_1^{n,d} \le \theta)\right) .
$$

\noindent Let $p(T,x,\theta)$ be the density of the diffusion, and $p_n(T,x,\theta)$ the density of the Euler scheme at time $T$. 
The derivative w.r.t. $\theta$ of the previous equality is:
$$\forall k \ge 1, \forall (\theta,x)\in \R\times \R^d, \ 
\frac{d^k}{d\theta^k}h(\theta) - \frac{d^k}{d\theta^k}h^n(\theta) = \frac{1}{1-\ell}\left(\frac{\partial^{k-1}}{\partial\theta^{k-1}}p^{X^d_1}(1,x,\theta) -\frac{\partial^{k-1}}{\partial\theta^{k-1}}p_n^{X^{n,d}_1}(1,x,\theta) \right),
$$

\noindent where we denote by $p^{X^d_1}(1,x,\theta)$ and $p_n^{X^{n,d}_1}(1,x,\theta)$ the marginal densities of $X^d_1$ and
$X_1^{n,d}$. Consequently, we observe that in order to apply Theorem \ref{implic:discret:err:main_result}, we have to give an expansion of the marginal densities and their derivatives, up to an order $k>1$. Actually, we will show that the expansion holds for $p(1,x,\theta) - p_n(1,x,\theta)$ and its derivatives, the expansion for the marginals will follow from an integration over the $d-1$ first components.
 
\subsection{Expansion for the densities.}

Using a continuity technique known as the Parametrix expansion, Konakov and Mammen \cite{kona:mamm:02}, in the Brownian case, and 
Konakov and Menozzi \cite{kona:meno:10}, in the stable case, successfully derive an expansion for the density of the solution of \eqref{EqDfSt} to an arbitrary order, with explicit terms. The purpose of this section is to extend these results to the derivatives of the densities.

The Parametrix expansion consists in representing the density of the solution of \eqref{EqDfSt} 
as a series involving the density of a \textit{frozen} equation and the generators associated with \eqref{EqDfSt} and the \textit{frozen} density.
We take a few lines here to describe this technique.

We define the following process as the frozen process.
Recall $T=1$ is a fixed deterministic time. For a given terminal point $y \in \R^d$, the frozen equation 
at point $y$ is defined as:
\begin{equation}\label{FrozenProc}
\tilde{X}_t = x + b(y)t + \sigma(y)Z_t.
\end{equation}
Thanks to the uniform ellipticity of $\sigma$, the process \eqref{FrozenProc} has a density with respect to the Lebesgue measure. Recalling that $\Sigma(z) = \sigma(z)\sigma(z)^T$, the density is given by:
\begin{eqnarray*}
\tilde{p}_\alpha^y(t,x,y)= 
\begin{dcases}
 \frac{\det(\Sigma(y))^{-1/2} }{(2\pi t)^{d/2}} 
\exp \left( -\frac{1}{2t}(y-x-b(y)t)^T \Sigma(y)^{-1} (y-x-b(y)t) \right), &{\rm if} \  \alpha =2\\
 \frac{1}{(2\pi)^d} \int_{\R^d}dp e^{-i\langle p, y-x-b(y)t\rangle} \exp\left( -t \int_{S^{d-1}} |\langle p,\sigma(y)\vartheta \rangle|^\alpha \mu(d\vartheta)\right), &{\rm if } \ \alpha \in (0,2).
 \end{dcases}
\end{eqnarray*}
We will often drop the superscript $y$ with the convention $\tilde{p}_\alpha(t,x,y)=\tilde{p}_\alpha^y(t,x,y)$, when no ambiguity is possible.
The distance between $p(t,x,y)$ and $\tilde{p}_\alpha(t,x,y)$ will then be quantified by the difference of the generators of \eqref{EqDfSt} and \eqref{FrozenProc}. The generator of the SDE \eqref{EqDfSt}:
\begin{eqnarray*}
Lf(t,x,y) =
\begin{dcases}
 \frac{1}{2} \tr \left(\Sigma(x) \partial^2_x f(t,x,y) \right) +\langle b(x),\partial_x f(t,x,y)\rangle, &\mbox{ if $\alpha=2$ },\\
\langle b(x) , \partial_x f(t,x,y) \rangle +\int_{\R^d} f(t,x+ \sigma(x)z,y) -f(t,x,y) - \frac{\langle \nabla_x f(t,x,y), \sigma(x)z \rangle}{1+ |z|^2}\nu(dz), &\mbox{ if $\alpha\in(0,2)$}.
\end{dcases}
\end{eqnarray*}

Let us define the generator of the frozen process \eqref{FrozenProc}:
\begin{eqnarray*}
\tilde{L}^*f(t,x,y) = 
\begin{dcases}
\frac{1}{2} \tr \left(\Sigma(y)\partial^2_xf(t,x,y) \right)+ \langle b(y),\partial_x f(t,x,y) \rangle, &\mbox{ if  $\alpha =2$,}\\
\langle b(y) , \partial_x f(t,x,y) \rangle +\int_{\R^d} f(t,x+ \sigma(y)z,y) -f(t,x,y) - \frac{\langle \nabla_x f(t,x,y), \sigma(y)z \rangle}{1+ |z|^2}\nu(dz), &\mbox{ if $\alpha \in (0,2)$}.
\end{dcases}
\end{eqnarray*}
When $\alpha=2$, these are differential operators of order 2.
For $\alpha\in (0,2)$, these operators should be seen as fractional derivative of order $\alpha$.
\begin{THM}\label{parametrix_series}
Under the assumptions \textbf{[A]}, the solution of \eqref{EqDfSt} exists and has density with respect to the Lebesgue measure.
Let $p(t,x,y)$ denote the density of \eqref{EqDfSt}. It admits the following representation:
$$
p(t,x,y) = \sum_{k=0}^{\infty} \tilde{p}_\alpha \otimes H^{(k)}(t,x,y),
$$
where we denoted $H(t,x,y) = (L-\tilde{L}^*)\tilde{p}_\alpha(t,x,y)$, and $\otimes$ is the space-time convolution:
$$
f\otimes g\left(t,x,y\right) = \int_0^t\int_{\R^d}f\left(u,x,z\right) g\left(t-u,z,y\right)dzdu,
$$
and $H^{(k)} (t,x,y) = H^{(k-1)} \otimes H (t,x,y), \mbox{ and } \tilde{p}_\alpha \otimes H^{(0)} (t,x,y)=\tilde{p}_\alpha(t,x,y)$.
\end{THM}

This result has been investigated in the literature, let us mention Friedman \cite{friedman:64} for the Brownian case and Kolokoltsov \cite{kolo:00} for the stable case.
The proof relies on a precise study of the frozen density and its derivatives (fractional derivatives in the stable case), and show that in the time space convolution, the time singularities induced by the derivation can be compensated to get a convergent series.

Similarly, one gets an equivalent result for the density of the Euler scheme.
We introduce the "frozen Markov chains" $(\tilde{X}_{t_k}^n)_{k\in \leftB 0,n \rightB}$:
$$
\tilde{X}_{t_k}^n =x, \   \tilde{X}_{t_{k+1}}^n = \tilde{X}_{t_k}^n + b(y) \Delta + \sigma(y)(Z_{t_{k+1}}- Z_{t_k}).
$$

We denote the discrete generators:
\begin{eqnarray}
L_nf(t_k-t_j,x,y) &=& \Delta^{-1}\left(\int p_n(\Delta,x,z) f(t_k-t_{j+1},z,y)dz - f(t_k-t_{j+1},x,y)\right),\label{GEN_DISC_1}\\
\tilde{L}_n^*f(t_k-t_j,x,y) &=& \Delta^{-1}\left(\int \tilde{p}^y(\Delta,x,z) f(t_k-t_{j+1},z,y)dz - f(t_k-t_{j+1},x,y)\right)\label{GEN_DISC_2}.
\end{eqnarray}


We then obtain a representation of the density of the Euler scheme using the frozen density and the discrete generators.
\begin{THM}\label{parametrix_euler}
The density $p_n(t_k,x,y)$ of the Euler scheme admits the following representation:
$$
p_n(t_k-t_j,x,y) = \sum_{r=0}^{k-j} \tilde{p}_\alpha \otimes_n H_n^{(r,n)}(t_k-t_j,x,y),
$$

\noindent where we denoted $H_n(t_k,x,y) = (L_n -\tilde{L}_n) \tilde{p}(t_k,x,y)$, and $\otimes_n$ is the discretized space-time convolution:
$$
f\otimes_n g\left(t_k,x,y\right) = \frac{1}{n}\sum_{i=0}^{k-1} \int_{\R^d}f\left(t_i,x,z\right) g\left(t_k-t_i,z,y\right)dz,
$$
and $H_n^{(r,n)} (t_k,x,y) = H_n^{(r-1,n)} \otimes_n H_n (t_k,x,y), \mbox{ where } \tilde{p}_\alpha \otimes H_n^{(0,n)} (t_k,x,y)=\tilde{p}_\alpha(t_k,x,y)$.
\end{THM}

\begin{REM}
We use the notation $H_n^{(r,n)} (t_k,x,y)$ to emphasize the dependency in the discretization of the convolution.
That is, the subscript $n$ refers to the discrete generators, whereas the super script $(r,n)$ refers respectively to the number of steps we iterate the convolution, and the number of discretization dates. Therefore, we have $H_n^{(1,n)} (t_k,x,y)= H_n(t,x,y)$.
Also, using the convention $H_n^{(r,n)}=0$ for $r >k-j$, we can write $p_n(t_k-t_j,x,y) = \sum_{r=0}^{+\infty} \tilde{p}_\alpha \otimes_n H_n^{(r,n)}(t_k-t_j,x,y)$.
\end{REM}

Once again, these results have been investigated in the literature and we state them here without proof.
The reader may consult \cite{kona:mamm:02, kona:meno:10} and the references therein.

Roughly speaking, we see that the differences between the two expansions of Theorems \ref{parametrix_series} and \ref{parametrix_euler} come from the convolution and the kernel.
Thus, in order to get an expansion for $p-p_n$, we introduce for all $k\in \leftB 0,n-1 \rightB$:
\begin{eqnarray*}
p^d(t_k,x,y) &=& \sum_{r=0}^{+\infty} \tilde{p}_\alpha \otimes_n H^{(r,n)} (t_k,x,y) , \\
H^{(r,n)} (t_k,x,y) &=& H^{(r-1,n)} \otimes_n H (t_k,x,y), \mbox{ where } \tilde{p}_\alpha \otimes H^{(0,n)} (t_k,x,y)=\tilde{p}_\alpha(t_k,x,y).
\end{eqnarray*}

Formally speaking, $p^d$ is the series of Theorem \ref{parametrix_series}, with discretized time integrals .
We then look for an expansion for the two differences $p-p_n = p-p^d + p^d-p_n$.
To that end, we define $\tilde{L}_*f(t,x,y) = \tilde{L}_x f(t,x,y)$, where :
\begin{eqnarray*}
\tilde{L}_{\xi} f(t,x,y)=
\begin{dcases}
 \frac{1}{2} \tr \left( \Sigma(\xi) \partial^2_x f(t,x,y) \right) + \langle b(\xi),\partial_x f(t,x,y)\rangle, &\mbox{ if $\alpha =2$},\\
\langle b(\xi) ,\partial_x f(t,x,y) \rangle - \int_{S^{d-1}} | \langle\partial_x,\sigma(\xi)\vartheta \rangle |^\alpha f(t,x,y)\mu (d\vartheta), &\mbox{ if $\alpha \in (0,2)$}.
\end{dcases}
\end{eqnarray*}

Note that both generators $\tilde{L}^*$ and $\tilde{L}_*$ depends on the freezing parameter $y$.
This induces extra caution below, as we will be led to differentiate with respect to the freezing parameter.

Extending the results of Theorem $1.1$ in Konakov and Mammen in \cite{kona:mamm:02}, for the Brownian case, and Theorem $21$ in Konakov and Menozzi in \cite{kona:meno:10}, for the Stable case, we have the following result. 
\begin{THM}\label{ExpDerv}
Assume that \textbf{[A]} holds.
Let $M\in\N^*$ be such that when $\alpha=2$, $0<M \le m/2$, and when $\alpha <2$, we assume $m > d+4$ and $0<M\le m-(d+4)$. Let $\gamma \in \N^d$, with $|\gamma|\le M$. 
Then, for all $x,y \in \R^d$, we have:
\begin{eqnarray} \label{EXP_DERIV_ DENS}
\partial_y^\gamma p(1,x,y)-\partial_y^\gamma p_n(1,x,y) &=& \sum_{k=1}^{M-1-|\gamma|} \frac{1}{(k+1)!n^k}
\partial_y^\gamma \left(
p\otimes_n \big(L-\tilde{L}^*\big)^{k+1}p^d \right)(1,x,y) \\
&&\qquad - 
\frac{1}{(k+1)!n^k} \partial_y^\gamma  \left(p^d\otimes_n \big(\tilde{L}_*-\tilde{L}^*\big)^{k+1} p_n\right)(1,x,y) 
+ \frac{\partial_y^\gamma R(x,y)}{n^{M-|\gamma|}}.\nonumber
\end{eqnarray}
Also, there is a constant $C>0$ depending on the set of assumptions \textbf{[A]}, $T$, $\gamma$, and $M$ such that the following bound holds for each term and the remainders:
\begin{eqnarray}
\sum_{k=1}^{M-|\gamma|-1} \left|\partial_y^\gamma\left(p\otimes_n \big(L-\tilde{L}^*\big)^{k+1}p^d\right)(1,x,y) \right| +  \left| \partial_y^\gamma \left(p^d\otimes_n \big(\tilde{L}_*-\tilde{L}^*\big)^{k+1} p_n \right)(1,x,y)\right| \nonumber \\
+ |\partial_y^\gamma R(x,y)|
\leq C\bar{p}^\alpha_K(t,x,y),\label{GAUSSIAN_ESTIMATE}
\end{eqnarray}
where for a given $K>0$, we denoted $\bar{p}^\alpha_K(t,x,y)$ the following quantity:
\begin{eqnarray*}
\bar{p}^\alpha_K(t,x,y) =
\begin{dcases}
 t^{-d/2} \exp \left(-K\frac{|y-x|^2}{t} \right), \  &\rm{if} \ \alpha=2,\\
 \frac{t^{-d/\alpha}}{  \left[K \vee \frac{|y-x|}{t^{\frac{1}{\alpha}}}\right]^{d+\alpha}}, \ & \rm{if} \ \alpha \in (0,2).
 \end{dcases}
\end{eqnarray*}

\end{THM}

For $\gamma=0$, expansion \eqref{EXP_DERIV_ DENS} is given in \cite{kona:mamm:02} in the Brownian case, and in \cite{kona:meno:10} in the stable case.
To get an expansion for $\partial_y^\gamma(p-p_n)(1,x,y)$,
we take the derivative along $y$ in each term in that expansion, and
prove that each one is bounded by an $\alpha$ stable density.

Formally, $\bar{p}^\alpha_K(t,x,y)$ is a stable density (up to some normalizing constant depending on $K>0$).
Observe that $\bar{p}^\alpha$ satisfies a semi-group property in the following sense:
\begin{PROP}\label{SEMI_GROUP_PROP_PK}
For all $\tau \in (0,t)$, for all $x,y\in \R^d$ for all $K_1,K_2>0$, there exists $K,C>0$ depending on the set of assumptions $\textbf{[A]}$ and the terminal time $T$, such that:
\begin{equation}\label{SemiGroup}
\int_{\R^d}\bar{p}^\alpha_{K_1}(\tau,x,z)\bar{p}^\alpha_{K_2}(t-\tau,z,y)dz \le C \bar{p}^\alpha_K(t,x,y).
\end{equation}
\end{PROP}
\begin{proof}
Indeed, for all $\alpha \in (0,2]$, we have that for $t>0$, for all $x,y \in \R^d$, there exists $c,C,K>0$ such that:
\begin{equation}\label{DOMINATION_DENS_GEL}
c \bar{p}^\alpha_K(t,x,y) \le \tilde{p}_\alpha^y(t,x,y) \le C \bar{p}^\alpha_K(t,x,y).
\end{equation}

For the gaussian case, we refer to the seminal paper \cite{friedman:64} or Sheu \cite{sheu:91} for a stochastic control based approach. For the stable case $\alpha<2$, the reader may consult and Kolokolstov \cite{kolo:00}.
Thus, one easily gets:
$$
\int_{\R^d}\bar{p}^\alpha_{K_1}(\tau,x,z)\bar{p}^\alpha_{K_2}(t-\tau,z,y)dz \le C
\int_{\R^d}\tilde{p}^y_\alpha(\tau,x,z)\tilde{p}^y_\alpha(t-\tau,z,y)dz  =  C \tilde{p}^y_\alpha(t,x,y) \le C \bar{p}^\alpha_K(t,x,y).
$$
\end{proof}

Using the previous density, we are able to bound the various terms appearing above.

\begin{lem}\label{LemmeDerivDens}
For all multi index $\gamma, \eta\in \N^d$ such that $|\gamma|+|\eta| \le m$ if $\alpha >1$, and $|\gamma|+|\eta| \le m-1$ if $\alpha \le 1$, for all $x,y \in \R^d$, for all  $t\in[0,T]$, for all $k\in \leftB 0,n-1 \rightB$, there exists $C= C(\textbf{[A]},T,\gamma,\eta)>0$ such that the following bounds holds:
\begin{eqnarray}
|\partial_x^\gamma \partial_y^\eta p^d(t_k,x,y)| +|\partial_x^\gamma \partial_y^\eta p_n(t_k,x,y)| \le C t_k^{-\frac{|\gamma|+|\eta|}{\alpha}}\bar{p}^\alpha_K(t_k,x,y),\label{DerivDens1}\\
|\partial_x^\gamma \partial_y^\eta p(t,x,y)|\le C t^{-\frac{|\gamma|+|\eta|}{\alpha}}\bar{p}^\alpha_K(t,x,y),\label{DerivDens2}
\end{eqnarray}

Moreover, for all $\xi \in \R^d$, 
\begin{equation}\label{DerivDensDeux}
|\partial_x^\gamma p^d(t_k,x,x+\xi)| +|\partial_x^\gamma p_n(t_k,x,x+\xi)| \le C \bar{p}^\alpha_K(t_k,x,x+\xi).
\end{equation}

Eventually, when $\alpha <2$, denoting $\Phi(t_k,x,y) = \sum_{r=1}^\infty H^{(r,n)}(t_k,x,y)$, we have:
\begin{eqnarray}
\left|\partial_x^\gamma \partial_y^\eta\Phi(t_k,x,y) \right| &\le& C t_k^{-\frac{|\gamma|+|\eta|}{\alpha}} \bar{p}^\alpha_K(t_k,x,y) \left(1+ \frac{1 \wedge|x-y|}{t_k} \right), \label{DerivPhi}\\
\left|\partial_x^\gamma \Phi(t_k,x,x+\xi) \right| &\le& C \bar{p}^\alpha_K(t_k,x,x+\xi) \left( 1+ \frac{1 \wedge|\xi|}{t_k}  \right).\label{DerivPhiDeux}
\end{eqnarray}

\end{lem}

\begin{REM}
We point out that in equations \eqref{DerivDensDeux} and \eqref{DerivPhiDeux}, 
despite the presence of derivations, there are no singularities induced by them, as the derivation argument appears in both the forward and the backward arguments.
This will be a key point in the proof of Theorem \ref{ExpDerv}.
\end{REM}

\begin{proof}
For the Brownian case, all the above estimates are proved in \cite{kona:mamm:02}.
We thus focus on the stable case.
In Konakov Menozzi \cite{kona:meno:10}, the bound \eqref{DerivDens1} and \eqref{DerivDens2} are given.
To get the bound \eqref{DerivDensDeux}, we prove \eqref{DerivPhi} and \eqref{DerivPhiDeux}, using the following estimates proved in \cite{kona:meno:10}:
\begin{eqnarray}
\left|\partial_x^\gamma \partial_y^\eta H(t,x,y) \right| &\le& C_1 t^{-\frac{|\gamma|+|\eta|}{\alpha}} \bar{p}^\alpha_K(t,x,y) \left(1+ \frac{1 \wedge|x-y|}{t} \right),\label{DerivH} \\
\left|\partial_x^\gamma H(t,x,x+\xi) \right| &\le& C_1 \bar{p}^\alpha_K(t,x,x+\xi) \left( 1+ \frac{1 \wedge|\xi|}{t}  \right). \label{DerivHDeux}
\end{eqnarray}

We then derive \eqref{DerivDensDeux} for the derivative of the densities using the expansion:
\begin{equation}\label{DefPhiSum}
p^d(1,x,y) = \tilde{p}_\alpha(t,x,y) + \frac{1}{n}\sum_{i=0}^{n-1}\int_{\R^d} \tilde{p}_\alpha \left(t_i,x,z \right) \Phi \left(1-t_i,z,y \right) dz.
\end{equation}

To get the bound on $p_n$, one may proceed similarly.
Denoting $\Phi_n(t_k,x,y) = \sum_{r=1}^\infty H^{(r,n)}_n(t_k,x,y)$, we investigate its derivatives and prove
$
\left|\partial_x^\gamma \partial_y^\eta\Phi_n(t_k,x,y) \right| \le C t_k^{-\frac{|\gamma|+|\eta|}{\alpha}} \bar{p}^\alpha_K(t_k,x,y) \left(1+ \frac{1 \wedge|x-y|}{t_k} \right)$
and 
$\left|\partial_x^\gamma \Phi_n(t_k,x,x+\xi) \right| \le C \bar{p}^\alpha_K(t_k,x,x+\xi) \left( 1+ \frac{1 \wedge|\xi|}{t_k}  \right)$, by proving a similar estimate to \eqref{DerivH} and \eqref{DerivHDeux} with $H_n$ instead of $H$. The estimate on $p_n$ will then be given by the counter part of representation \eqref{DefPhiSum} for $p_n$. We do not enter into the computational details.


We begin with \eqref{DerivPhiDeux}.
Observe that due to the presence of the derivation parameter in both the forward and the backward arguments, the derivatives does not yield any additional singularities.
From \eqref{DerivHDeux}, we prove by induction the following:
\begin{equation}\label{DerivHN}
\left|\partial_x^\gamma H^{(r,n)}(t_k,x,x+\xi) \right| \le C_{r} t_k^{(r-1)\omega} \bar{p}^\alpha_K(t_k,x,x+\xi) \left(1 +\frac{1\wedge|\xi|}{t_k}\right),
\end{equation}
where $\omega = \frac{1}{\alpha}\wedge \alpha$, and the sequence of constants $(C_r)_{r \ge 0}$ is defined recursively by:
$$C_{r+1} = C_\gamma C_r C  \max \Big( \frac{1}{r\omega} , B\big((r-1)\omega+1,\omega \big)\Big) , \ C_1>0,$$
where $C_1$ is the constant appearing in bounds \eqref{DerivHDeux} and \eqref{DerivHN}, and $C$ is a positive constant independent of $r,\gamma, x,\xi$.
For $r=1$, the bound is exactly \eqref{DerivHDeux}.
Suppose that it holds for $r\ge1$.
We have using the induction hypothesis, equation \eqref{DerivHDeux} and Leibnitz's formula:
\begin{eqnarray}
\left|\partial_x^\gamma H^{(r+1,n)}(t_k,x,x+\xi) \right| &\le& \sum_{\eta=0}^\gamma C^\eta_\gamma \frac{1}{n} \sum_{i=0}^{k-1} \int_{\R^{d}} \left|  \partial_x^\eta H^{(r,n)}(t_i,x,z+x) \right| \left|  \partial_x^{\gamma-\eta} H(t_k-t_i,z+x,x+\xi) \right|  dz \nonumber\\
&\le& C_\gamma C_r C\frac{1}{n} \sum_{i=0}^{k-1} \int_{\R^{d}} t_i^{(r-1)\omega} \bar{p}^\alpha_K(t_i,x,x+z)\left(1 +\frac{1\wedge|z|}{t_i}\right) \nonumber\\
&&\quad \times \ \ \bar{p}^\alpha_K(t_k-t_i,x+z,x+\xi) \left( 1+ \frac{1 \wedge|\xi-z|}{t_k-t_i}  \right) dz. \label{ConvNCoup}
\end{eqnarray}

We decompose, the integral: 
\begin{eqnarray}
\int_{\R^{d}} \bar{p}^\alpha_K(t_i,x,x+z)\left(1 +\frac{1\wedge|z|}{t_i}\right) \bar{p}^\alpha_K(t_k-t_i,x+z,x+\xi) \left( 1+ \frac{1 \wedge|\xi-z|}{t_k-t_i}  \right)  dz\label{ConvDeuxCoup} = I_1+I_2+I_3+I_4
\end{eqnarray}
%
where:
\begin{eqnarray*}
I_1&=&\int_{\R^{d}} \bar{p}^\alpha_K(t_i,x,x+z) \bar{p}^\alpha_K(t_k-t_i,x+z,x+\xi)  dz,\\
I_2&=&\int_{\R^{d}} \bar{p}^\alpha_K(t_i,x,x+z)\bar{p}^\alpha_K(t_k-t_i,x+z,x+\xi)  \frac{1 \wedge|\xi-z|}{t_k-t_i}  dz\\
I_3&=&\int_{\R^{d}} \bar{p}^\alpha_K(t_i,x,x+z)\frac{1\wedge|z|}{t_i} \bar{p}^\alpha_K(t_k-t_i,x+z,x+\xi)  dz,\\
I_4&=&\int_{\R^{d}} \bar{p}^\alpha_K(t_i,x,x+z)\frac{1\wedge|z|}{t_i}\bar{p}^\alpha_K(t_k-t_i,x+z,x+\xi)  \frac{1 \wedge|\xi-z|}{t_k-t_i}   dz.
\end{eqnarray*}

The first one $I_1$ is bounded by $C\bar{p}^\alpha_K(t_k,x,x+\xi)$ thanks to the semi-group property (Proposition \ref{SEMI_GROUP_PROP_PK}). 
By symmetry, $I_2$ and $I_3$ are treated the same way. We focus on $I_2$.
In the rest, we denote by the symbol $\asymp$ the relation:
$$
f(x) \asymp g(x) \Leftrightarrow \exists C>1,\  \forall x \in \R^{d}:\   C^{-1} g(x) \le f(x) \le C g(x).
$$
We argue differently, according to the ratio $|\xi|/t_k^{1/\alpha}$. 
\begin{itemize}
\item Suppose first that $|\xi| \le C t_k^{1/\alpha}$.
Then, the diagonal estimate holds: $\bar{p}^\alpha_K(t_k,x,x+\xi)\asymp t_k^{-d/\alpha}$.
On the one hand, if $i \ge k/2$, then $t_i \asymp t_k$. Since the diagonal estimate is a global bound, one has:
$\bar{p}^\alpha_K(t_i,x,x+z)\le C t_i^{-d/\alpha} \asymp C t_k^{-d/\alpha} \asymp C \bar{p}^\alpha_K(t_k,x,x+\xi) $.
On the other hand, when $i \le k/2$, then $t_k-t_i \asymp t_k$, and we have
$\frac{1}{t_k-t_i}\bar{p}^\alpha_K(t_k-t_i,x+z,x+\xi) \le C \frac{1}{t_k-t_i} (t_k-t_i)^{-d/\alpha}  \asymp C \frac{1}{t_k}\bar{p}^\alpha_K(t_k,x,x+\xi)$.
\item Suppose now that $|\xi| \ge C t_k^{1/\alpha}$. 
Then, the off-diagonal estimate holds: $\bar{p}^\alpha_K(t_k,x,x+\xi)\asymp \frac{t_k}{|\xi|^{d+\alpha}}$.
Now, since $|\xi| \le |z| + |\xi-z|$, we have either $1/2|\xi| \le |z|$, or $1/2|\xi| \le |\xi-z|$.
In the first case the off-diagonal estimate holds for the first density:
$\bar{p}^\alpha_K(t_i,x,x+z)\asymp \frac{t_i}{|z|^{d+\alpha}} \le C \frac{t_k}{|\xi|^{d+\alpha}} \asymp C \bar{p}^\alpha_K(t_k,x,x+\xi)$.
In the second case, the second density is off-diagonal and we can write:
$\frac{1}{t_k-t_i}\bar{p}^\alpha_K(t_k-t_i,x+z,x+\xi) \le C \frac{1}{t_k-t_i}\frac{t_k-t_i}{|\xi-z|^{d+\alpha}}\le \frac{1}{t_k} \frac{t_k}{|\xi|^{d+\alpha}} \asymp C \frac{1}{t_k} \bar{p}^\alpha_K(t_k,x,x+\xi)$.
\end{itemize}

Therefore, we always have the alternative:
\begin{equation}\label{ALTERNATIVE}
 \bar{p}^\alpha_K(t_i,x,x+z) \le C \bar{p}^\alpha_K(t_k,x,x+\xi) 
 \mbox{, or } \frac{1}{t_k-t_i} 
 \bar{p}^\alpha_K(t_k-t_i,x+z,x+\xi) \le C \frac{1}{t_k}\bar{p}^\alpha_K(t_k,x,x+\xi).
 \end{equation}
 
Combining this alternative with the smoothing effect of the Parametrix kernel $H$ reflected in the bound (see Section 3 of Kolokoltsov \cite{kolo:00}):
\begin{equation}\label{smoothing}
\forall \tau\in(0,T), \ \forall y\in\R^d,  \  \int_{\R^d} \frac{1 \wedge|y-z|}{\tau} \bar{p}^\alpha_K(\tau,z,y) dz \le \tau^{(\alpha \wedge 1)-1},
\end{equation}
gives that the second and third terms are bounded by: 
$$
I_2+I_3
  \le C \bar{p}^\alpha_K(t_k,x,x+\xi)\left(t_i^{(\alpha \wedge 1)-1}+ (t_k-t_i)^{(\alpha \wedge 1)-1} +\frac{1\wedge|\xi|}{t_k}\right).
$$

We now turn to the last term in \eqref{ConvDeuxCoup}, that writes:
$$
I_4=\int_{\R^{d}} \frac{1\wedge|z|}{t_i}\bar{p}^\alpha_K(t_i,x,x+z)\frac{1 \wedge|\xi-z|}{t_k-t_i}  \bar{p}^\alpha_K(t_k-t_i,x+z,x+\xi)  dz.
$$
When $\bar{p}^\alpha_K(t_k,x,x+\xi)$ is in the diagonal regime, that is, when $|\xi|\le Ct_k^{\frac{1}{\alpha}}$, we have:
$$
 \frac{1\wedge|z|}{t_i}\bar{p}^\alpha_K(t_i,x,x+z) \le C t_i^{-d/\alpha} t_i^{\frac{1}{\alpha}-1}\ {\rm and } \  \frac{1\wedge|\xi-z|}{t_k-t_i}\bar{p}^\alpha_K(t_k-t_i,x+z,x+\xi) \le C (t_k-t_i)^{-d/\alpha} (t_k-t_i)^{\frac{1}{\alpha}-1}.
$$
We prove the first inequality, the second one is obtained with the same arguments.
Let us assume first that $|z|\le Ct_i^{1/\alpha}$. In that  case the diagonal estimate holds for $\bar{p}^\alpha_K(t_i,x,x+z)$, thus:
$$
\frac{1\wedge|z|}{t_i}\bar{p}^\alpha_K(t_i,x,x+z) \le \frac{1\wedge|z|}{t_i}t_i^{-d/\alpha} 
\le \frac{|z|}{t_i} t_i^{-d/\alpha} \le C  t_i^{1/\alpha-1}\times t_i^{-d/\alpha}.
$$
On the other hand, when the off-diagonal estimate holds for $\bar{p}^\alpha_K(t_i,x,x+z)$, that is when $|z| >Ct_i^{1/\alpha}$, we have:
$$
\frac{1\wedge|z|}{t_i}\bar{p}^\alpha_K(t_i,x,x+z) \le \frac{1\wedge|z|}{t_i} \frac{t_i}{|z|^{d+\alpha}} \le C \frac{1}{|z|^{d+\alpha-1}} \le Ct_i^{-\frac{1}{\alpha}(d+\alpha-1)}=Ct_i^{-\frac{d}{\alpha}} t_i^{\frac{1}{\alpha}-1}.
$$
Thus, in both cases, we obtained the announced bound.

Now, if $i \le k/2$, $t_k \asymp t_k-t_i$, one has: 
$$
\frac{1\wedge|\xi-z|}{t_k-t_i}\bar{p}^\alpha_K(t_k-t_i,x+z,x+\xi) \le C \bar{p}^\alpha_K(t_k,x,x+\xi) (t_k-t_i)^{\frac{1}{\alpha}-1}.
$$
Then, using \eqref{smoothing}, $I_4\le \bar{p}^\alpha_K(t_k,x,x+\xi)  (t_k-t_i)^{\frac{1}{\alpha}-1} t_i^{(\alpha \wedge 1)-1}$.
Similarly, when $i > k/2$, we use that $t_k \asymp t_i$, to get 
$$
\frac{1\wedge|z|}{t_i}\bar{p}^\alpha_K(t_i,x,x+z) \le C \bar{p}^\alpha_K(t_k,x,x+\xi) t_i^{\frac{1}{\alpha}-1}.
$$
Consequently, when $|\xi| \le C t_k^{\frac{1}{\alpha}}$, $I_4$ is bounded in the following way:
$$
I_4 \le C \bar{p}^\alpha_K(t_k,x,x+\xi) \Big((t_k-t_i)^{\frac{1}{\alpha}-1} t_i^{(\alpha \wedge 1)-1} + t_i^{\frac{1}{\alpha}-1} (t_k-t_i)^{(\alpha \wedge 1)-1}\Big).
$$

Assume now that $|\xi| > C t_k^{\frac{1}{\alpha}}$. In that case using similar arguments one may prove that we have either:
$$
\frac{1\wedge|\xi-z|}{t_k-t_i}\bar{p}^\alpha_K(t_k-t_i,x+z,x+\xi) \le C \frac{1\wedge|\xi|}{t_k}\bar{p}^\alpha_K(t_k,x,x+\xi),
$$
or:
$$
\frac{1\wedge|z|}{t_i}\bar{p}^\alpha_K(t_i,x,x+z) \le C \frac{1\wedge|\xi|}{t_k}\bar{p}^\alpha_K(t_k,x,x+\xi).
$$
Thus, using \eqref{smoothing}, $I_4$ is now bounded as follows:
$$
I_4 \le  C\frac{1\wedge|\xi|}{t_k}\bar{p}^\alpha_K(t_k,x,x+\xi) \Big(t_i^{(\alpha \wedge 1)-1} + (t_k-t_i)^{(\alpha \wedge 1)-1} \Big).
$$

Plugging this estimate in \eqref{ConvDeuxCoup} in turn implies:
\begin{eqnarray}\label{CTRL_INT_INNER}
\int_{\R^{d}} \bar{p}^\alpha_K(t_i,x,x+z)\left(1 +\frac{1\wedge|z|}{t_i}\right) \bar{p}^\alpha_K(t_k-t_i,x+z,x+\xi) \left( 1+ \frac{1 \wedge|y-z|}{t_k-t_i}  \right)  dz  \nonumber\\
\le C \bar{p}^\alpha_K(t_k,x,x+\xi) \Biggl( \Big((t_k-t_i)^{\frac{1}{\alpha}-1} t_i^{(\alpha \wedge 1)-1} + t_i^{\frac{1}{\alpha}-1} (t_k-t_i)^{(\alpha \wedge 1)-1}\Big) \nonumber \\
+ \frac{1\wedge|\xi|}{t_k}\Big(t_i^{(\alpha \wedge 1)-1} + (t_k-t_i)^{(\alpha \wedge 1)-1} \Big) \Biggr).
\end{eqnarray}

Plugging bound \eqref{CTRL_INT_INNER} in \eqref{ConvNCoup} yields:

\begin{eqnarray}\label{CTRL_DERIV_H_N}
\left|\partial_x^\gamma H^{(r+1,n)}(t_k,x,x+\xi) \right|
&\le& C_\gamma C_r C  \bar{p}^\alpha_K(t_k,x,x+\xi) \frac{1}{n} \sum_{i=0}^{k-1}  t_i^{(r-1)\omega}    \Big((t_k-t_i)^{\frac{1}{\alpha}-1} t_i^{(\alpha \wedge 1)-1} + t_i^{\frac{1}{\alpha}-1} (t_k-t_i)^{(\alpha \wedge 1)-1}\Big) \nonumber \\
&&+ C_\gamma C_r C  \frac{1\wedge|\xi|}{t_k}\bar{p}^\alpha_K(t_k,x,x+\xi) \frac{1}{n} \sum_{i=0}^{k-1}t_i^{(r-1)\omega} \Big(t_i^{(\alpha \wedge 1)-1} + (t_k-t_i)^{(\alpha \wedge 1)-1} \Big)  . 
\end{eqnarray}

Now, assume first that $\alpha \ge 1$.
Recalling that $\omega = \frac{1}{\alpha}\wedge \alpha = \frac{1}{\alpha}$, the above bound becomes:

\begin{align*}
\left|\partial_x^\gamma H^{(r+1,n)}(t_k,x,x+\xi) \right|
&\le C_\gamma C_r C  \bar{p}^\alpha_K(t_k,x,x+\xi) \frac{1}{n} \sum_{i=0}^{k-1}  t_i^{(r-1)\frac{1}{\alpha}}    \Big((t_k-t_i)^{\frac{1}{\alpha}-1} + t_i^{\frac{1}{\alpha}-1}\Big) \nonumber \\
& + C_\gamma C_r C  \frac{1\wedge|\xi|}{t_k}\bar{p}^\alpha_K(t_k,x,x+\xi) \frac{1}{n} \sum_{i=0}^{k-1}t_i^{(r-1)\frac{1}{\alpha}}\\
&\le C_{r+1}  t_k^{\frac{r}{\alpha}} \bar{p}^\alpha_K(t_k,x,x+\xi) \left(1 + \frac{1\wedge|\xi|}{t_k} \right),
\end{align*}

\noindent where 
$C_{r+1} = C_\gamma C C_r \max \left( B\left((r-1)\frac{1}{\alpha}+1,\frac{1}{\alpha}\right) , \frac{\alpha}{r}\right)$ and we used that $t_k^{(r-1)\frac{1}{\alpha}+1}\le t_k^{\frac{r}{\alpha}}$, since $\alpha \ge 1$ and $t_k\le1$.

On the other hand, when $\alpha \le 1$, $\omega=\alpha$ one similarly proves that:
%

\begin{eqnarray*}
\left|\partial_x^\gamma H^{(r+1,n)}(t_k,x,x+\xi) \right|
\le C_{r+1} t_k^{r\alpha} \bar{p}^\alpha_K(t_k,x,x+\xi) \left( 1 + \frac{1\wedge|\xi|}{t_k}  \right),
\end{eqnarray*}
with $C_{r+1} = C_\gamma C_r C  \max\Big( \frac{1}{r\alpha} , B\big((r-1)\alpha+1,\alpha \big)\Big) $.
This constant is coherent with the previous one, setting $C_{r+1} = C_\gamma C_r C  \max \Big( \frac{1}{r\omega} , B\big((r-1)\omega+1,\omega \big)\Big) $.
This concludes the proof of bound \eqref{DerivHN}.
%
%
%
%
Observe that by definition of Euler's Beta function, $(C_r)_{r \ge 0}$ produces a convergent series.
To get the bound  \eqref{DerivPhiDeux}, we sum bounds \eqref{DerivHN}. In order to get the bound \eqref{DerivDensDeux}, we now plug the bound \eqref{DerivPhiDeux} in equation \eqref{DefPhiSum}, and from similar arguments, one derives \eqref{DerivDensDeux}.

To prove \eqref{DerivPhi}, we show by induction the following bound:
\begin{equation}
\left| \partial_x^\gamma \partial_y^\eta H^{(r,n)}(t_k,x,y) \right| \le C_r t_k^{(r-1)\omega-\frac{|\gamma|+|\eta|}{\alpha}}C_1 \bar{p}^\alpha_K(t_k,x,y) \left(1+ \frac{1 \wedge|x-y|}{t_k} \right).
\end{equation}
For $r=1$, this bound is exactly \eqref{DerivH}.
To get the estimate for $r+1$, we proceed as above.
$$
\left| \partial_x^\gamma \partial_y^\eta H^{(r+1,n)}(t_k,x,y) \right| = \left| \partial_x^\gamma \partial_y^\eta \frac{1}{n}\sum_{i=0}^{k-1} \int_{\R^d} H^{(r,n)}(t_i,x,z)H(t_k-t_i,z,y) dz \right| \le I+II,
$$
where
\begin{eqnarray*}
I= \left| \partial_x^\gamma \partial_y^\eta \frac{1}{n}\sum_{i\le k/2} \int_{\R^d} H^{(r,n)}(t_i,x,z)H(t_k-t_i,z,y)  dz \right|, \ 
II= \left| \partial_x^\gamma \partial_y^\eta \frac{1}{n}\sum_{i\ge k/2} \int_{\R^d} H^{(r,n)}(t_i,x,z)H(t_k-t_i,z,y)  dz \right|.
\end{eqnarray*}

In $I$, the time parameter $t_i$ is small, thus, the singularities induced by the derivation of $H^{(r,n)}(t_i,x,z)$ are the worst. In order to get rid of them, we make use of a change of variable to get:
\begin{eqnarray*}
 I=\left| \partial_x^\gamma \frac{1}{n}\sum_{i\le k/2}\int_{\R^d} H^{(r,n)}(t_i,x,z)\partial_y^\eta H(t_k-t_i,z,y)  dz \right|
 =\left| \partial_x^\gamma  \frac{1}{n}\sum_{i\le k/2} \int_{\R^d} H^{(r,n)}(t_i,x,z+x)\partial_y^\eta H(t_k-t_i,z+x,y)  dz \right|.
\end{eqnarray*}

\noindent Now, from equations \eqref{DerivHN}, \eqref{DerivH} and Leibnitz's formula we derive:

\begin{eqnarray*}
I &\le&\sum_{\beta=0}^\gamma C^\beta_\gamma \frac{1}{n}\sum_{i\le k/2} \int_{\R^d} \left| \partial_x^\beta H^{(r,n)}(t_i,x,z+x) \right| \left| \partial_x^{\gamma-\beta}\partial_y^\eta H(t_k-t_i,z+x,y)\right|  dz\\
&\le& C_\gamma C_n \frac{1}{n}\sum_{i\le k/2} \int_{\R^d} t_i^{(r-1)\omega} \left(1+\frac{1\wedge|z|}{t_i} \right) \bar{p}^\alpha_K(t_i,x,x+z) \\
&&\times (t_k-t_i)^{-\frac{|\gamma|+|\eta|}{\alpha}} \left(1+\frac{1\wedge|y-x-z|}{t_k-t_i} \right) \bar{p}^\alpha_K(t_k-t_i,x+z,y)  dz.\\
&\le& C_{n+1} t_k^{r\omega-\frac{|\gamma|+|\eta|}{\alpha}} \bar{p}^\alpha_K(t_k,x,y) .
\end{eqnarray*}
where we used that $t_k \asymp t_k-t_i$ for $i \le k/2$ for the last inequality. Note that once again, the series $\sum_{r\ge1} C_r$ converges. For $II$, we proceed with similar arguments. In this case, we use the change variables $w=z+y$ instead.

\end{proof}

\begin{proof}[Proof of Theorem \ref{ExpDerv}.]

The coefficients are the sum of two terms.
We only focus on the first term, the second term can be treated similarly.
From the definition of $\otimes_n$, we have:

$$
\partial_y^\gamma p\otimes_n \big(L-\tilde{L}^*\big)^{k+1}p^d(1,x,y)
=\partial_y^\gamma \frac{1}{n} \sum_{i=0}^{n-1}\int_{\R^d} p(t_i,x,z) 
\big(L-\tilde{L}^*\big)^{k+1}p^d(1-t_i,z,y) dz.
$$

To deal with the singularities coming from the derivatives we split the sum over $i$ in two parts:

\begin{eqnarray*}
\partial_y^\gamma p\otimes_n \big(L-\tilde{L}^*\big)^{k+1}p^d(1,x,y)
&=&\partial_y^\gamma \frac{1}{n} \sum_{i < n/2}\int_{\R^d} p(t_i,x,z) 
\big(L-\tilde{L}^*\big)^{k+1}p^d(1-t_i,z,y) dz \label{SUM_1}\\
&&+
\partial_y^\gamma \frac{1}{n} \sum_{i\ge n/2}\int_{\R^d} p(t_i,x,z) 
\big(L-\tilde{L}^*\big)^{k+1}p^d(1-t_i,z,y) dz \label{SUM_2}\\
&=&S_1 + S_2.
\end{eqnarray*}

For $S_1$, the time parameter is small, thus the singularities brought by the derivation in $y$ and the generators are negligible.
Indeed, exchanging the derivation and the integral:
$$S_1
=  \frac{1}{n} \sum_{i < n/2}\int_{\R^d} p(t_i,x,z) 
\partial_y^\gamma \left( \big(L-\tilde{L}^*\big)^{k+1}p^d(1-t_i,z,y) \right) dz.
$$
 
 From bound \eqref{DerivDens1} in Lemma \ref{LemmeDerivDens}, we derive:
 

 \begin{eqnarray*}
\left|\partial_y^\gamma \left( \big(L-\tilde{L}^*\big)^{k+1}p^d (1-t_i,z,y)\right)\right| &\le& C\left(1-t_i\right)^{-k-1-\frac{\gamma}{\alpha}} \bar{p}^\alpha_K\left(1-t_i,z,y\right).
\end{eqnarray*}
The right hand side of the previous equation is bounded uniformly in $y$, thus, from the Lebesgue theorem, we can derive under the integral.
Now, since $p(t_i,x,z)\le C \bar{p}^\alpha_K(t_i,x,z)$, this sum yields by a semi-group property:
\begin{eqnarray*}
\left|\partial_y^\gamma \frac{1}{n} \sum_{i < n/2}\int_{\R^d} p(t_i,x,z) 
\big(L-\tilde{L}^*\big)^{k+1}p^d(1-t_i,z,y) dz \right|
&\le& C \frac{1}{n} \sum_{i < n/2}
\int_{\R^d} \bar{p}^\alpha_K(t_i,x,z)
\left(1-t_i\right)^{-\frac{\gamma}{\alpha}-k-1} \bar{p}^\alpha_K\left(1-t_i,z,y\right).\\
& =& C\bar{p}^\alpha_K(1,x,y).
\end{eqnarray*}
We now turn to the second sum.
When $i \ge n/2$, by an integration by parts it follows
\begin{eqnarray*}
S_2&=&\partial_y^\gamma \frac{1}{n} \sum_{i\ge n/2}\int_{\R^d} p(t_i,x,z) 
\big(L-\tilde{L}^*\big)^{k+1}p^d(1-t_i,z,y) dz\\
&=&\partial_y^\gamma \frac{1}{n} \sum_{i\ge n/2}\int_{\R^d} \Big(\big(L-\tilde{L}^*\big)^{k+1} \Big)^T p(t_i,x,z) 
p^d(1-t_i,z,y) dz.
\end{eqnarray*}

\noindent where $\Big(\big(L-\tilde{L}^*\big)^{k+1} \Big)^T$ stands for the adjoint of $\big(L-\tilde{L}^*\big)^{k+1}$, which is well defined thanks to the smoothness of the coefficients $b$ and $\sigma$. The operator $L-\tilde{L}^*$ is an integro-differential operator (a derivative of order $\alpha$), so that
the operator $\Big(\big(L-\tilde{L}^*\big)^{k+1} \Big)^T$ is still an integro-differential operator which yields singularity of the same order as $(L-\tilde{L}^*)^{k+1}$, thus, applied to $p$ yields singularities which are still negligible since $i\geq n/2$. However, the derivative $\partial_y^\gamma$ will affect $p^d(1-t_i,z,y)$, thus, giving additional singularities so that beforehand we make use of the change of variable: $z=y-u$ to derive
\begin{eqnarray*}
S_2 = \frac{1}{n} \sum_{i\ge n/2} \sum_{\eta =0}^\gamma C_\gamma^\eta \int_{\R^d}  \partial_y^{\gamma-\eta} \left[\Big(\big(L-\tilde{L}^*\big)^{k+1} \Big)^T p\left(t_i,x,y-u\right) \right]\partial_y^\eta p^d\left(1-t_i,y-u,y\right) du.
\end{eqnarray*}
Now, bound \eqref{DerivDensDeux} of Lemma \ref{LemmeDerivDens} gives:  $\forall \eta \in \N^{*}$,  $\exists C>0$ s.t.: $
\left| \partial_y^{\eta} p^d\left(1-t_i,y-u,y\right) \right| \le C \bar{p}^\alpha_K\left(1-t_i,y-u,y\right)$.
On the other hand, since $i\ge n/2$, the singularities of the derivatives on the first density are negligible.
We thus get from a semi group property:

\begin{eqnarray*}
|S_2| \le  \frac{C}{n} \sum_{i\ge n/2} \sum_{\eta =0}^\gamma C_\gamma^\eta 2^{\frac{|\gamma|+|\eta|}{\alpha}+k+1} \int_{\R^d}  \bar{p}^\alpha_K\left(t_i,x,y-u\right) \bar{p}^\alpha_K\left(1-t_i,y-u,y\right)du
\le C_\gamma \bar{p}^\alpha_K\left(1,x,y\right).
\end{eqnarray*}

In order to prove that the expansion \eqref{EXP_DERIV_ DENS} makes sense, it remains to prove the bound on the remainder.
Since the expansion \eqref{EXP_DERIV_ DENS} is made of two contributions $p-p^d$ and $p^d-p_n$, the remainder $R(x,y)$ also splits in two terms $R(x,y) = R_1(1,x,y) + R_2(1,x,y)$, each being a remainder respectively of the expansion of $p-p^d$ and $p^d-p_n$, with:
\begin{eqnarray*}
R_1(1,x,y) &=&  \sum_{r\ge 0} (Q_M\otimes_n H^{(r,n)})(1,x,y),\\
Q_M(t_i,x,y)&=& \frac{1}{M!} \sum_{i=0}^{k-1}\int_{i/n}^{(i+1)/n} \left[ n\left( u-i/n\right)\right]
\int_0^1(1-\delta)^{M-1}  \int \frac{\partial^M}{\partial s ^M} \left[ p(s,x,z) H(t_i-s,z,y)\right]_{s=t_i+\delta(u-t_i)} dz d\delta du,
\end{eqnarray*}
and
\begin{eqnarray*}
R_2(t,x,y)&=& \frac{1}{(M+1)!} \int_0^1 (1-\tau)^M \biggl[p^d \otimes_n \Big(\tilde{L}_* - \tilde{L}^* \Big)^{M+1}\tilde{p}^\Delta_\tau \biggr](t,x,y)\\
\tilde{p}^\Delta_\tau(t,x,y) &=& \sum_{r\ge0} \tilde{p}_\tau \otimes_n H^{(r,n)}_n(t,x,y); \ H_n(t,x,y) = (L_n-\tilde{L}_n^*)\tilde{p}_\alpha(t ,x,y),
\end{eqnarray*}
where $\tilde{p}_\tau(t,x,y) = \int_{\R^{d}} \tilde{p}_\alpha^x(\tau \Delta,x,z)\tilde{p}_\alpha^y(t-\tau \Delta , z,y )dz$ and $L_n$ and $\tilde{L}_n^*$ stands for the discrete generators defined in equations \eqref{GEN_DISC_1} and \eqref{GEN_DISC_2}.
The reader may consult Konakov and Mammen \cite{kona:mamm:02} and Konakov and Menozzi \cite{kona:meno:10} for more details.

We focus on the estimation of $R_1$. 
The arguments for $R_2$ are similar and hinted at the end of this proof.
Observe that we can write:
\begin{eqnarray}\label{CONV_R_M}
R_1(1,x,y) &=& Q_M(1,x,y) + \frac{1}{n}\sum_{i=0}^{n-1} \int Q_M \left(t_i,x,z \right) \Phi \left(1-t_i,z,y\right)dz,\\
\Phi\left(1-t_i,z,y\right) &=& \sum_{r = 1}^{+\infty} H^{(r,n)}\left(1-t_i,z,y\right).
\end{eqnarray}

Using Kolmogorov's forward and backward equations, one shows by induction on $M$ that $Q_M(t,x,y)$ can be written as:
\begin{eqnarray*}
Q_M(t_k,x,y)&=& \frac{1}{M!} \sum_{i=0}^{k-1}\int_{t_i}^{(i+1)/n} \left[ n\left( u-t_i\right)\right]
\int_0^1(1-\delta)^{M-1}\\
&& \int p\left( t_i+\delta \left(u- t_i\right),x,z\right) (L-\tilde{L}^*)^{M+1} \tilde{p}_\alpha\left(t_k- (t_i+\delta \left(u- t_i\right)),z,y \right) dz d\delta du.
\end{eqnarray*}
See e.g. equation (4.5) in \cite{kona:mamm:02} for the brownian case and equation (4.4) in \cite{kona:meno:10} for the stable case.
Thus, apart from the additional integrations w.r.t. $\delta$ and $u$, we see that $Q_M$ is of the same nature that the terms we already dealt with. Therefore, adapting the above arguments allows us to derive $\alpha$-stable estimates for this term.
In fact, a precise study of $Q_M$ allows us to derive that  $\forall i \in \leftB1,n \rightB$, $\forall\gamma\ge 0$, $\exists C>0$ such that:
\begin{eqnarray}\label{EST_DERIV_R_M}
\left| \partial_y^\gamma Q_M\left(t_k,x,y\right)\right| \le C t_k^{-\frac{\gamma}{\alpha}-(M+1)}\bar{p}^\alpha_K\left(t_k,x,y\right).
\end{eqnarray}
To get this bound, we actually show that it holds for 
$$
\bar{Q}_M\left(t_k,x,y\right) = \int p\left(t,x,z\right) (L-\tilde{L}^*)^{M+1} \tilde{p}_\alpha\left(t_k-t,z,y \right) dz,
$$
independently of $t \in [0,1]$. The arguments differs depending if $t$ is closer to 0 or $t_k$.
In the first case (say, $t\le t_k/2$), the singularities induced by taking the derivative along $y$ in $\tilde{p}\left(t_k-t,z,y \right)$ are bounded by to $t_k^{-\frac{\gamma}{\alpha}-(M+1)}$, which is the announced singularity.
When $t$ is closer to $t_k$ (say $t> t_k/2$), we transfer the operator $(L-\tilde{L}^*)^{M+1}$ on $p$ by taking the adjoint, and we change variables to $z=y-u$. The singularities in $\partial_y^\gamma\Big( (L-\tilde{L}^*)^{M+1}\Big)^T p\left(t,x,y-u\right)$ then yields the announced $t_k^{-\frac{\gamma}{\alpha}-(M+1)}$, and we conclude using the semi-group property of Proposition \ref{SEMI_GROUP_PROP_PK}.
Thus, for $t_k=1$, we have $|\partial_y^\gamma Q_M(1,x,y)|\le C \bar{p}^\alpha_K(1,x,y)$, which is the announced bound.

Now, for the second part of \eqref{CONV_R_M}, we split the sum:
\begin{align*}
\frac{1}{n}\sum_{i=0}^{n-1} \int dz Q_M \left(t_i,x,z \right) \Phi \left(1-t_i,z,y\right) &= \frac{1}{n}\sum_{i\le \frac{n}{2}} \int  dz Q_M \left(t_i,x,z \right) \Phi \left(1-t_i,z,y\right)\\
& + \frac{1}{n}\sum_{i\ge \frac{n}{2}} \int  dz Q_M \left(t_i,x,z \right) \Phi \left(1-t_i,z,y\right) \\
& = S_1(1,x,y) + S_2(1,x,y).
\end{align*}

In $S_1(1,x,y)$, the time parameter is such that when differentiating along $y$, the singularities are negligible. Thus, we derive under the integral, and use bounds \eqref{EST_DERIV_R_M} and \eqref{DerivPhi}, to get a convolution of $\bar{p}^\alpha_K$ functions.
For $S_2(1,x,y)$, we make use of the change of variable $z = y-u$, to get:

$$
 S_2(1,x,y) = \frac{1}{n}\sum_{i\ge \frac{n}{2}} \int  du Q_M \left(t_i,x,y-u \right) \Phi \left(1-t_i,y-u,y\right).
$$

In the stable case, when taking the derivative along $y$ the bounds \eqref{EST_DERIV_R_M} and \eqref{DerivPhiDeux} yield:
\begin{eqnarray*}
\left|\partial_y^\gamma S_2\left(1,x,y\right)\right| &=& \sum_{\eta=0}^\gamma C^\gamma_\eta \frac{1}{n}\sum_{i\ge \frac{n}{2}} \int  du  
\left|\partial_y^{\gamma-\eta}Q_M \left(t_i,x,y-u \right) \right| \left| \partial_y^\eta \Phi \left(1-t_i,y-u,y\right) \right|\\
&\le& C_\gamma \frac{1}{n}\sum_{i\ge \frac{n}{2}} \int  du \bar{p}^\alpha_K(t_i,x,y-u) \left(1+ \frac{1\wedge|u|}{1-t_i} \right)\bar{p}^\alpha_K(1-t_i,y-u,y)\\
&\le& C\bar{p}^\alpha_K(1,x,y)+ \frac{1}{n}\sum_{i\ge \frac{n}{2}} \int  du \bar{p}^\alpha_K(t_i,x,y-u) \left( \frac{1\wedge|u|}{1-t_i} \right)\bar{p}^\alpha_K(1-t_i,y-u,y).
\end{eqnarray*}
Now, recall from definition of $\bar{p}^\alpha_K$, we have either $ \bar{p}^\alpha_K(t_i,x,z) \le C  \bar{p}^\alpha_K(1,x,y)$ or $\frac{1}{1-t_i} \bar{p}^\alpha_K(1-t_i,z,y) \le C \bar{p}^\alpha_K(1,x,y)$, so that:
$$
\frac{1}{n}\sum_{i\ge \frac{n}{2}} \int  du \bar{p}^\alpha_K(t_i,x,y-u) \left( \frac{1\wedge|u|}{1-t_i} \right)\bar{p}^\alpha_K(1-t_i,y-u,y) \le \bar{p}^\alpha_K(1,x,y)\left( 1+ 1\wedge|y-x| \right).
$$
In the gaussian case, the proof is simpler, as the derivative of $\Phi$ is estimated by:
$$
\left| \partial_y^\eta \Phi \left(1-t_i,y-u,y\right) \right| \le \frac{1}{\sqrt{1-t_i}} \bar{p}^\alpha_K(1-t_i,y-u,y),
$$
and we can directly conclude comparing the sum over $\eta$ to a Beta function.

For $R_2$, we can take the derivative in $y$ under the integral in the above expression to get:
\begin{eqnarray*}
\partial_y^\gamma R_2(t,x,y)&=& \frac{1}{(M+1)!} \int_0^1 (1-\tau)^M \partial_y^\gamma \biggl[p^d \otimes_n \Big(\tilde{L}_* - \tilde{L}^* \Big)^{M+1}\tilde{p}^\Delta_\tau \biggr](t,x,y),
\end{eqnarray*}
which is  a term of the same nature as the second part of the expansion \eqref{EXP_DERIV_ DENS}.
In particular, one can show bounds on $\tilde{p}^\Delta_\tau(t,x,y)$, similar to those of Lemma \ref{LemmeDerivDens}.
With these estimates at hand, we may use similar arguments to derive an $\alpha$-stable estimate on $R_2$, for $\alpha \in (0,2]$.
We leave the remaining details to the reader.

\end{proof}

\begin{REM}\label{Rem_Ind_N}
The terms in the expansion \eqref{EXP_DERIV_ DENS} depends on $n$.
As already pointed out in \cite{kona:meno:10} and \cite{kona:mamm:02}, it is possible to make this expansion independent of $n$, using the bounds on the difference between the usual time space convolution $\otimes$ and its discretization $\otimes_n$.
For $M=2$, one derives the expansion:
\begin{eqnarray*}
\partial_y^\gamma(p-p_n)(1,x,y) &=& \frac{1}{2n} \partial_y^\gamma \Big( p \otimes_n (L-\tilde{L}^*)^{2}p^d \Big)(1,x,y)\\
&&-\frac{1}{2n} \partial_y^\gamma \Big( p^d \otimes_n (\tilde{L}_*-\tilde{L}^*)^{2}p_n \Big)(1,x,y) + \frac{1}{n^2}\partial_y^\gamma R(x,y)\\
&=& \frac{1}{2n} \partial_y^\gamma \Big( p \otimes (L-\tilde{L}^*)^{2}p \Big)(1,x,y)\\
&&-\frac{1}{2n} \partial_y^\gamma \Big( p \otimes (\tilde{L}_*-\tilde{L}^*)^{2}p \Big)(1,x,y) + \frac{1}{n^2}\partial_y^\gamma \tilde{R}(x,y)\\
&=&\frac{1}{2n} \partial_y^\gamma \Big( p \otimes (L^2-(\tilde{L}^*)^2)p \Big)(1,x,y)+ \frac{1}{n^2}\partial_y^\gamma \tilde{R}(x,y).
\end{eqnarray*}
In the above expansion, $\partial_y^\gamma \tilde{R}(x,y)$ is a remainder term bounded by some stable density as $\partial_y^\gamma R(x,y)$.
\end{REM}

\begin{COROL} \label{EXP_THETA}
Assume \textbf{[A]} holds.
Recall $m$ denotes the regularity of the coefficients of the SDE \eqref{EqDfSt}.
Let $M\in\N^*$, such that when $\alpha=2$, $0<M \le m/2$, and when $\alpha <2$, we assume $m>d+4$ and $0<M\le m-(d+4)$. 
Then, the following expansion holds:
$$\theta^{*,n}- \theta^* = \frac{C_1}{n} + \cdots + \frac{C_p}{n^{M-1}} + o\left( \frac{1}{n^{M-1}}\right).$$
\end{COROL}

\begin{proof}
We prove that under the assumptions of Corollary \ref{EXP_THETA}, \textbf{[H-(M-1)]} holds.
Let $M\le m/2$ for $\alpha =2$ and $M \le m-(d+4)$ for $\alpha <2$ and let $\gamma \in \N$ with $\gamma \le M-1$.
From Theorem \ref{ExpDerv}, under \textbf{[A]}, expansion \eqref{EXP_DERIV_ DENS} holds up to order $M-1$.
Moreover, from Remark \ref{Rem_Ind_N}, this expansion can be made independent of $n$, namely
\begin{eqnarray}\label{EXP_DERIV_IND_N}
\partial_{y^d}^\gamma p(1,x,y)-\partial_{y^d}^\gamma p_n(1,x,y) = 
\sum_{k=1}^{M-\gamma-1} \frac{1}{n^k} \Gamma_k^\gamma(x,y)+ \frac{\partial_y^\gamma \tilde{R}_n(x,y)}{n^{M-\gamma}}.
\end{eqnarray}
Thus, integrating equation \eqref{EXP_DERIV_IND_N} in the $d-1$ first variables yields for all $M\le m/2$ for $\alpha =2$ and $M \le m-(d+4)$ for $\alpha <2$ and $\gamma \le M-1$:
\begin{equation}\label{EXP_DERIV_MARG}
\partial_{y^d}^\gamma p^{X^{d}_1}(1,x,y^d)-\partial_{y^d}^\gamma p_n^{X^{n,d}_1}(1,x,y^d) = \sum_{k=1}^{M-\gamma-1} \frac{1}{n^k} \bar{\Gamma}_k^\gamma(x,y^d)+ \frac{\partial_{y^d}^\gamma \bar{R}_n^\gamma (x,y^d)}{n^{M-\gamma}}.
\end{equation}

\noindent where we denoted $p^{X^{d}_1}$,  $p_n^{X^{n,d}_1}$ the marginal densities of $X_1^d$ and $X_1^{n,d}$, and 
$$
\bar{\Gamma}_k^\gamma(x,y^d) = \int_{\R^d \times \cdots \times \R^d} \Gamma_k^\gamma(x,y) dy_1 \cdots dy_{d-1} .
$$

The Gaussian bound on the remainder implies $\frac{\partial_{y_d}^l \bar{R}_n(x,y^d)}{n^{M-\gamma}} = \O\left(n^{-(M-\gamma)}\right)$, so that we have for all $M\le m/2$ for $\alpha =2$ and $M \le m-(d+4)$ for $\alpha <2$ and $\gamma \le M-1$:
\begin{equation}\label{dev_dens_marg}
\partial_{y^d}^\gamma p^{X^{d}_1}(1,x,y^d)-\partial_{y^d}^\gamma p_n^{X^{n,d}_1}(1,x,y^d) = \sum_{k=1}^{M-\gamma-1} \frac{1}{n^k} \bar{\Gamma}_k^\gamma(x,y^d)+ o(\frac{1}{n^{M-\gamma-1}}).
\end{equation}

Now, since $h(\theta) -  h^n(\theta) = \mathbb{P}^x (X_1^d \le \theta)  - \mathbb{P}^x (X_1^{n,d} \le \theta) $, taking $\gamma=0$ in \eqref{dev_dens_marg} yields:
\begin{eqnarray*}
h(\theta) -  h^n(\theta) = \int_{-\infty}^{\theta} \left(p^{X^{d}_1}(1,x,y_d)- p_n^{X^{n,d}_1}(1,x,y_d)\right) dy_d
=\sum_{k=1}^{M-1} \frac{\Lambda_k^0(\theta)}{n^k} +o(\frac{1}{n^{M-1}}),
\end{eqnarray*}

\noindent where we denoted
$\Lambda_k^0(\theta)=\int_{-\infty}^{\theta} \bar{\Gamma}_k^0(x,y_d)dy_d $.
Thus, the first assumption in \textbf{[H-(M-1)]} holds.

We now turn to the expansion of the derivatives.
From expansion \eqref{dev_dens_marg} one easily gets:
$$
\partial_\theta (h-h^n)(\theta) = (p^{X^{d}_1}- p_n^{X^{n,d}_1})(1,x,\theta) =  \sum_{k=1}^{M-2} \frac{1}{n^k} \bar{\Gamma}_k^0(x,\theta)+ o(\frac{1}{n^{M-2}})
$$

\noindent and $\forall l \le M-2, \forall (x,\theta)\in \R^d\times \R$
\begin{eqnarray*}
\partial_{\theta}^l h(\theta) - \partial_{\theta}^l h^n(\theta) & = &  \partial_{\theta}^{l-1} p^{X^{d}_1}(1,x,\theta) -\partial_{\theta}^{l-1} p_n^{X^{n,d}_1}(1,x,\theta)  =  \sum_{k=1}^{M-l-1} \frac{1}{n^k}\Lambda_k^l(\theta)+ o(\frac{1}{n^{M-1-l}})
\end{eqnarray*}

\noindent where we denoted for consistency $\Lambda_k^l(\theta)=\bar{\Gamma}_k^{l-1}(x,\theta)$.
Consequently, expansions \eqref{erreur_discretisation} and \eqref{dev_diff_h} holds in $\textbf{[H-(M-1)]}$. 
It remains to check the local uniform convergence and the invertibility of $Dh(\theta^*)$.
For the latter, recall that $Dh(\theta^*)= p^d(1,x,\theta^*)$.
Also, we know that under \textbf{[A]}, stable bounds holds for $p(1,x,y)$ (see e.g. \cite{Aronson:1967} in the gaussian case, and \cite{kolo:00} for the stable case), that is $
 p(1,x,y) \asymp \bar{p}^\alpha_K(1,x,y)$.
Thus, the left hand side of the previous inequality gives that $p(1,x,y)$, and \textit{a fortiori} $p^{X^d}(1,x,y)$,  is never equal to zero. Hence $Dh(\theta^*)$ is invertible.
Finally the local uniform convergence is a consequence of expansion \eqref{EXP_DERIV_ DENS}. 
\end{proof}
\section{Numerical illustration}
To illustrate the method we consider a geometric Brownian motion $(X_t)_{t \in [0,T]}$ with dynamics given by
$$
X_t = x_0 \exp((r-\sigma^2/2)T + \sigma \sqrt{T} W_t), \ t\in [0,T]
$$

\noindent for which the quantile is explicitly known at any level $\ell \in (0,1)$. Indeed a simple computation shows that 
$$
\theta^* = x_0 \exp((r-\sigma^2/2)T + \sigma \sqrt{T} \phi^{-1}(\ell))
$$

\noindent where $\phi$ is the distribution function of the standard normal distribution $\mathcal{N}(0,1)$. Let us note that the assumptions of Section \ref{applic:sec} are not satisfied in this example and that nobody would devise any kind of Monte Carlo simulation in practice since the law of $X_T$ is explicitly known for any time $T$. However the Black-Scholes model and its Euler scheme appears as a natural and often used benchmark to test and evaluate the performance of Monte Carlo methods. We use the following values for the parameters: $x_0 = 100$, $r=0.05$, $\sigma=0.4$, $T=1$, $\ell=0.7$. The reference Black-Scholes quantile is $\theta^*=119.69$. We set $\gamma(p) = \gamma_0/p $ with $\gamma_0=60$.

Let us note that in order to implement the Richardson-Romberg stochastic approximation estimator we need to simulate discretization schemes of the Brownian diffusion with different steps $\Delta_r = T/nr$, $r=1, \cdots, R$. We thus need to simulate \emph{consistent Brownian increments} on intervals of the form $[(k-1)T/(rn), kT/(rn)]$, $r=1, \cdots, R$. The coefficients to compute by induction the Brownian increments from small intervals up to the root interval of length $T/n$ have been computed up to $R=5$ for $\alpha=1$ and up to $R=3$ for $\alpha=1/2$ in \cite{GPag:07}, Section 5.

In order to illustrate the result of Theroem \ref{implic:discret:err:main_result}, we plot in Figure \ref{RRSA} the behaviors of $\sum_{r=1}^R \bold{w}_r \theta^{*,rn}-\theta^*$ for $R=2,3,4$ and $n=2, \cdots, 15$. We estimate $\theta^{*,rn}$ by $\theta^{rn}_M$, with $M=10^6$ samples for $R=2$ and $M=10^8$ samples for $R=3,4$ using the same Brownian motion for each R (see Remark \ref{control:var:rem}). We clearly see that the Richardson-Romberg estimator efficiency increases with $R$ and the method gives satisfying results with $R=3,4$ for small values of $n$. 

\begin{figure}[!ht]
\begin{center}
\includegraphics[width=10cm,height=10cm, keepaspectratio=true]{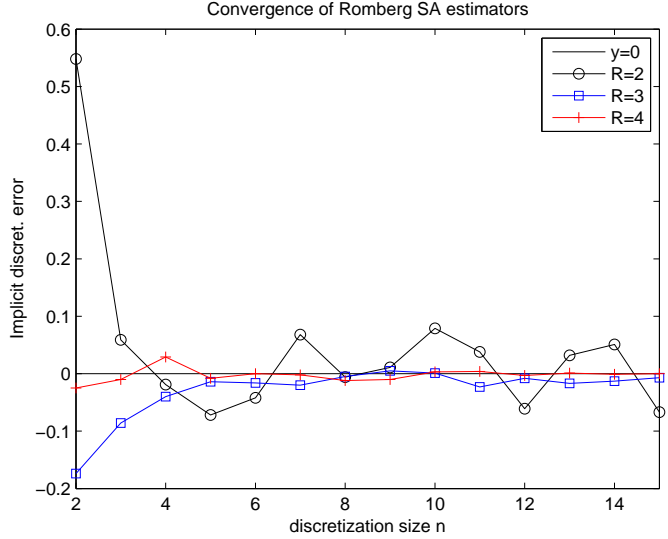}
\caption{Richardson Romberg SA estimators: $\sum_{r=1}^R \bold{w}_r \theta^{*,rn}-\theta^*$ with respect to $n=2,\cdots,15$ for $R=2,3,4$.}
\label{RRSA}
\end{center}
\end{figure}
Let us observe that the asymptotic optimal parameters in Propositions \ref{sub:opt:res} and \ref{optim:prob:crude} depend on the structural parameters: $\alpha$, $|C_R|$, $C(\gamma, \underline{\lambda})$, $\E\left[|H(\theta^*,U)|^2\right]$, $|C_1|$. Let us note that in Proposition \ref{premier_pas} and Theorem \ref{implic:discret:err:main_result} one may show that the constants $|C_R|$ writes $C_R= Dh(\theta^*)^{-1} \tilde{C}_R$. Here one has $Dh(\theta^*) = p(1,x,\theta^*)/(1-\ell)$ so that $|C_R|= |\tilde{C}_R| (1-\ell)/p(1,x,\theta^*)$. We estimate $p(1,x,\theta^*)$ by $p_n(1,x,\theta^{n}_M)\approx (\P_x(X^{n,d}_1 \leq \theta^{n}_M + \varepsilon) - \P_x(X^{n,d}_1 \leq \theta^{n}_M - \varepsilon))/2\varepsilon$ which in turn is approximated by the crude Monte Carlo estimator $(2M)^{-1} \sum_{k=1}^M \textbf{\mbox{1}}_{\left\{(X^{n,d}_1)^{k} \leq \theta^* + \varepsilon\right\}} - \textbf{\mbox{1}}_{\left\{(X^{n,d}_1)^{k} \leq \theta^* - \varepsilon\right\}}$ with $M=1000$, $n=100$, $\varepsilon=0.1$ leading to the value $Dh(\theta^*)=2.56 \times 10^{-2}$. Finally estimating $|\tilde{C_1}|$ for the crude SA estimator and $|\tilde{C}_R|$ for the Richardson-Romberg extrapolation method is a challenging task. Consequently we implement these methods in a blind way setting $|\tilde{C}_R|=1$ for every $R$. We also set $\underline{\lambda}=Dh(\theta^*)$ and $\gamma_0 = 1/\underline{\lambda}$. Note that we have $\E\left[|H(\theta^*,U)|^2\right] = \E\left[|1-(1-\ell)^{-1} \textbf{\mbox{1}}_{\left\{X^{d}_1 \geq \theta^{*} \right\}}|^2\right] = \ell/ (1-\ell)$. The optimal parameters for the Richardson-Romberg extrapolation method are set according to Proposition \ref{sub:opt:res} namely
$$
n(\varepsilon ) = \ceil*{ \left(\frac{2 \alpha R}{\beta}  + 1 \right)^{\frac{1}{\alpha R}}  \mu^{\frac{1}{\alpha R}}_R \varepsilon^{-\frac{1}{\alpha R}} } \ \ \mbox{and} \ \ M(\varepsilon) = \ceil*{ \gamma^{\frac{1}{\beta}}_0 \nu^{\frac{2}{\beta}}_R \left(1+ \frac{\beta}{2\alpha R}\right)^{\frac{2}{\beta}} \varepsilon^{-\frac{2}{\beta}}}. 
$$

The target accuracy $\varepsilon$ for the $L^{1}$-error has been set at $\varepsilon=2^{-p}$, $p=1, \cdots,4$. The $L^{1}$-error is estimated using $400$ runs of the algorithm. The results are summarized in Table \ref{tab:RR} for the Richardson-Romberg extrapolation SA method and in Table \ref{tab:crudeSA} for the crude SA method.\footnote{The computations were performed on a computer with 4 multithreaded(16) octo-core processors
(Intel(R) Xeon(R) CPU E5-4620 @ 2.20GHz).} Note that as expected the $L^{1}$-error is always lower than the specified $\varepsilon$ for our estimators. Using the Richardson-Romberg SA scheme instead of the crude SA method leads to a gain in terms of CPU-time varying from $12$ (for $\varepsilon=5.00 \times 10^{-1}$) to $66$ (for $\varepsilon=6.25 \times 10^{-2}$).
\begin{table}
\centering
\begin{tabular}{ | c| l | c | r || c | c | c| }
\hline
Target accuracy: $\varepsilon$ & $L^{1}$-error & time ($s$) & $R$ & $n$ & $M$ \\
\hline
  $5.00 \times 10^{-1}$  & $3.21 \times 10^{-1}$ & $0.9 \times 10^{1}$ & $2$ & $14$ & $8.69 \times 10^{5}$ \\
  $2.50 \times 10^{-1}$  & $4.80 \times 10^{-2}$ & $5.15 \times 10^{1}$ & $2$ & $20$& $3.48 \times 10^{6}$ \\
  $1.25 \times 10^{-1}$  & $4.32 \times 10^{-2}$ & $1.70 \times 10^{2}$  & $3$ &  $8$ & $1.21 \times 10^{7}$  \\
  $6.25 \times 10^{-2}$ & $3.48 \times 10^{-2}$  & $7.92 \times 10^{2}$ & $3$ & $10$ & $4.85 \times 10^{7}$ \\
  \hline
\end{tabular}
\caption{Richardson-Romberg SA estimators for the quantile at level $\ell$ of a geometric Brownian motion with a target accuracy $\varepsilon=2^{-p}$, $p=1,\cdots,4$.}
 \label{tab:RR}
\end{table}
\begin{table}

\centering
\begin{tabular}{ | c| l | c || r | c | c| }
\hline
Target accuracy: $\varepsilon$ & $L^{1}$-error & time ($s$) & $n$ & $M$ \\
\hline
  $5.00 \times 10^{-1}$  & $2.09 \times 10^{-1}$ & $1,09 \times 10^2$  & $235$ & $1.25 \times 10^{6}$ \\
  $2.50 \times 10^{-1}$  & $3.84 \times 10^{-2}$ & $8.18 \times 10^2$  & $469$& $5.01 \times 10^{6}$ \\
  $1.25 \times 10^{-1}$  & $3.48 \times 10^{-2}$ & $7.09 \times 10^{3}$  & $938$ & $2.00 \times 10^{7}$  \\
  $6.25 \times 10^{-2}$ & $2.91 \times 10^{-2}$  & $5.25 \times 10^{4}$ & $1876$ & $8.01 \times 10^{7}$ \\
  \hline
\end{tabular}
\caption{Crude SA estimators for the quantile at level $\ell$ of a geometric Brownian motion with a target accuracy $\varepsilon=2^{-p}$, $p=1,\cdots,4$.}
\label{tab:crudeSA}
\end{table}
 \section{Technical results}\label{technical:res:sec}
We provide here some useful technical results that are used repeatedly throughout the paper. For a proof the reader may refer to \cite{Frikha2013}.
\begin{LEMME}
\label{stepseq:tech:lemme}
Let $a,b>0$. Suppose that \A{HUA} is satisfied. Let $(\gamma_n)_{n\geq1}$ be a sequence satisfying \A{HS}. If $\gamma(t)=\gamma_0/t$, $t\geq1$, suppose $b \underline{\lambda} \gamma_0>a$. Let $(v_n)_{n\geq1}$ be a non-negative sequence. Then, for some positive constant $C:=C(\underline{\lambda},\gamma)$, one has
$$
\lim\sup_{n} \gamma^{-a}_{n} \sum_{k=1}^{n} \gamma^{1+a}_{k} ||\Pi_{k+1,n}||^b v_k  \leq C \lim\sup_{n} v_n,
$$

\noindent where $\Pi_{k,n}:= \prod_{j=k}^n (I_d-\gamma_j Dh(\theta^*))$, with the convention that $\Pi_{n+1,n}=I_d$.

\end{LEMME}

\begin{LEMME}
\label{sstrongerror:tech:lemme}
Let $(\theta^{n}_{p})_{p\geq0}$ be the scheme defined by \eqref{RM}, $\theta^{n}_0$ being independent of the innovation with $\sup_{n\geq1}\E|\theta^{n}_0|^2<+\infty$. Suppose that \A{HUA}, \A{HC1} and \A{HS} hold. Then, for some constant $C>0$, one has:
$$
\forall p \geq 1, \ \ \sup_{n\geq1}\E[|\theta^{n}_p-\theta^{*,n}|^2] \leq C \gamma(p)
$$

\end{LEMME}
%
%
%
%
%
%
%
%
%
%
%
\providecommand{\bysame}{\leavevmode\hbox to3em{\hrulefill}\thinspace}
\providecommand{\MR}{\relax\ifhmode\unskip\space\fi MR }
\providecommand{\MRhref}[2]{%
  \href{http://www.ams.org/mathscinet-getitem?mr=#1}{#2}
}
\providecommand{\href}[2]{#2}


\bibliographystyle{alpha}
\bibliography{bibli}

\end{document}